\documentclass[12pt]{amsart}
\usepackage{a4wide}
\usepackage[T1]{fontenc}
\usepackage[utf8x]{inputenc}
\usepackage[english]{babel}
\usepackage{amscd,amsmath,amsthm,amssymb,graphics}
\usepackage{lmodern,pst-node}
\usepackage{multicol}
\usepackage{amsfonts,amssymb,amscd,amsmath,enumitem,verbatim}
\usepackage[pdfencoding=auto,bookmarks=true,bookmarksnumbered=true]{hyperref}

\def\NZQ{\Bbb}               

\def\RR{{\NZQ R}}

\def\frk{\frak}               

\def\Phi{{\frk n}}
\def\Phi{{\frk N}}
\def\MH{{\mathcal H}}

\def\MM{{\mathcal M}}
\def\MG{{\mathcal G}}

\def\OR{{\mathrm{OR}}}
\def\opn#1#2{\def#1{\operatorname{#2}}} 
\opn\chara{char} \opn\length{\ell} \opn\pd{pd} \opn\rk{rk}
\opn\projdim{proj\,dim} \opn\injdim{inj\,dim} \opn\rank{rank}
\opn\depth{depth} \opn\grade{grade} \opn\height{height}
\opn\embdim{emb\,dim} \opn\codim{codim}

\opn\Tr{Tr} \opn\bigrank{big\,rank}
\opn\superheight{superheight}\opn\lcm{lcm}
\opn\trdeg{tr\,deg}
\opn\reg{reg} \opn\lreg{lreg} \opn\ini{in} \opn\lpd{lpd}
\opn\size{size}\opn\bigsize{bigsize}
\opn\cosize{cosize}\opn\bigcosize{bigcosize}
\opn\sdepth{sdepth}\opn\sreg{sreg}
\opn\link{link}\opn\fdepth{fdepth}
\opn\div{div} \opn\Div{Div} \opn\cl{cl} \opn\Cl{Cl}
\opn\Spec{Spec} \opn\Supp{Supp} \opn\supp{supp} \opn\Sing{Sing}
\opn\Ass{Ass} \opn\Min{Min}\opn\Mon{Mon} \opn\dstab{dstab} \opn\astab{astab}
\opn\Syz{Syz}
\opn\Ann{Ann} \opn\Rad{Rad} \opn\Soc{Soc}
\opn\Im{Im} \opn\Ker{Ker} \opn\Coker{Coker} \opn\Am{Am}
\opn\Hom{Hom} \opn\Tor{Tor} \opn\Ext{Ext} \opn\End{End}
\opn\Aut{Aut} \opn\id{id}

\opn\nat{nat}
\opn\pff{pf}
\opn\Pf{Pf} \opn\GL{GL} \opn\SL{SL} \opn\mod{mod} \opn\ord{ord}
\opn\Gin{Gin} \opn\Hilb{Hilb}\opn\sort{sort}
\opn\initial{init}
\opn\ende{end}
\opn\height{height}
\opn\type{type}
\opn\aff{aff} \opn\con{conv} \opn\relint{relint} \opn\st{st}
\opn\lk{lk} \opn\cn{cn} \opn\core{core} \opn\vol{vol}
\opn\link{link} \opn\star{star}\opn\lex{lex}
\opn\gr{gr}

\def\pot#1#2{#1[\kern-0.28ex[#2]\kern-0.28ex]}

\opn\dirlim{\underrightarrow{\lim}}
\opn\inivlim{\underleftarrow{\lim}}
\let\sect=\cap

\let\iso=\cong

\let\Sect=\bigcap

\let\to=\rightarrow

\def\Implies{\ifmmode\Longrightarrow \else
        \unskip${}\Longrightarrow{}$\ignorespaces\fi}
\def\implies{\ifmmode\Rightarrow \else
        \unskip${}\Rightarrow{}$\ignorespaces\fi}
\def\iff{\ifmmode\Longleftrightarrow \else
        \unskip${}\Longleftrightarrow{}$\ignorespaces\fi}

\let\:=\colon
\makeatletter
\CheckCommand*\refstepcounter[1]{\stepcounter{#1}%
\protected@edef\@currentlabel
{\csname p@#1\endcsname\csname the#1\endcsname}%
}
\renewcommand*\refstepcounter[1]{\stepcounter{#1}%
\protected@edef\@currentlabel
{\csname p@#1\expandafter\endcsname\csname the#1\endcsname}%
}
\def\labelformat#1{\expandafter\def\csname p@#1\endcsname##1}
\DeclareRobustCommand\Ref[1]{\protected@edef\@tempa{\ref{#1}}%
\expandafter\MakeUppercase\@tempa
}
\makeatother
										
\makeatletter
\newcommand{\numberlike}[2]{%
\expandafter\def\csname c@#1\endcsname{%
\expandafter\csname c@#2\endcsname}%
}
\makeatother

\def\DefaultNumberTheoremWithin{section}

\theoremstyle{plain}
\newtheorem{Lemma}{Lemma}
\numberwithin{Lemma}{\DefaultNumberTheoremWithin}
\labelformat{Lemma}{Lemma~#1}

\numberwithin{Claim}{\DefaultNumberTheoremWithin}
\numberlike{Claim}{Lemma}
\labelformat{Claim}{Claim~#1}
\newtheorem{Theorem}{Theorem}
\numberwithin{Theorem}{\DefaultNumberTheoremWithin}
\numberlike{Theorem}{Lemma}
\labelformat{Theorem}{Theorem~#1}
\newtheorem{Corollary}{Corollary}
\numberwithin{Corollary}{\DefaultNumberTheoremWithin}
\numberlike{Corollary}{Lemma}
\labelformat{Corollary}{Corollary~#1}
\newtheorem{Proposition}{Proposition}
\numberwithin{Proposition}{\DefaultNumberTheoremWithin}
\numberlike{Proposition}{Lemma}
\labelformat{Proposition}{Proposition~#1}
\newtheorem{Conjecture}{Conjecture}
\numberwithin{Conjecture}{\DefaultNumberTheoremWithin}
\numberlike{Conjecture}{Lemma}
\labelformat{Conjecture}{Conjecture~#1}
									
\theoremstyle{definition}

\numberwithin{Definition}{\DefaultNumberTheoremWithin}
\numberlike{Definition}{Lemma}
\labelformat{Definition}{Definition~#1}
														
\theoremstyle{definition}

\numberwithin{Question}{\DefaultNumberTheoremWithin}
\numberlike{Question}{Lemma}
\labelformat{Question}{Question~#1}
																			
\theoremstyle{definition}

\numberwithin{Problem}{\DefaultNumberTheoremWithin}
\numberlike{Problem}{Lemma}
\labelformat{Problem}{Problem~#1}

\theoremstyle{remark}
\newtheorem{Remark}{Remark}
\numberwithin{Remark}{\DefaultNumberTheoremWithin}
\numberlike{Remark}{Lemma}
\labelformat{Remark}{Remark~#1}
\theoremstyle{remark}

\numberwithin{Example}{\DefaultNumberTheoremWithin}
\numberlike{Example}{Lemma}

\let\epsilon\varepsilon
\let\kappa=\varkappa
\def\qed{\ifhmode\textqed\fi
      \ifmmode\ifinner\quad\qedsymbol\else\dispqed\fi\fi}
\def\textqed{\unskip\nobreak\penalty50
       \hskip2em\hbox{}\nobreak\hfil\qedsymbol
       \parfillskip=0pt \finalhyphendemerits=0}
\def\dispqed{\rlap{\qquad\qedsymbol}}

\opn\dis{dis}
\def\pnt{{\raise0.5mm\hbox{\large\bf.}}}

\opn\Lex{Lex}

\begin{document}
\title{On the ideal of orthogonal representations of a graph in $\mathbb{R}^2$}

\author {J\"urgen Herzog}

\address{Fachbereich Mathematik, Universit\"at Duisburg-Essen, Campus Essen, 45117
Essen, Germany} \email{juergen.herzog@uni-essen.de}

\author{Antonio Macchia}
\address{Fachbereich Mathematik und Informatik, Philipps-Universit\"at Marburg, Hans-Meerwein-Strasse 6, 35032 Marburg, Germany} \email{macchia.antonello@gmail.com}

\author{Sara Saeedi Madani}

\address{School of Mathematics,
Institute for Research in Fundamental Sciences (IPM), P.O. Box 19395-5746, Tehran, Iran} \email{sarasaeedim@gmail.com}

\author{Volkmar Welker}
\address{Fachbereich Mathematik und Informatik,
         Philipps-Universit\"at Marburg, Hans-Meerwein-Strasse 6, 35032 Marburg, Germany} \email{welker@mathematik.uni-marburg.de}

\begin{abstract}
  In this paper, we study orthogonal representations of simple graphs $G$ in $\RR^d$ from
  an algebraic perspective in case $d = 2$. Orthogonal representations of
  graphs, introduced by Lov\'asz, are maps from the vertex set to $\RR^d$
  where non-adjacent vertices are sent to orthogonal vectors. We exhibit algebraic
  properties of the ideal generated by the equations expressing this condition
  and deduce geometric properties of the variety of orthogonal embeddings for
  $d=2$ and $\RR$ replaced by an arbitrary field. In particular, we classify
  when the ideal is radical and provide a reduced primary decomposition if
  $\sqrt{-1} \not\in K$. This leads to a description of the
  variety of orthogonal embeddings as a union of varieties defined by prime ideals.
  In particular, this applies to the motivating case $K = \RR$.
\end{abstract}

\thanks{The second author was supported by Università degli Studi di Bari}
\thanks{The paper was written while the third author was visiting the Department of Mathematics of University Duisburg-Essen. She wants to express her thanks for its hospitality. The research of the third author was supported by a grant from IPM}
\thanks{The fourth author was partially supported by MSRI}

\maketitle

\section{Introduction}

Orthogonal representations of graphs were introduced by Lov\'{a}sz in
1979 \cite{Lo}. In \cite{Lo} and subsequent work it has been shown that they
are intimately related to important combinatorial properties of graphs
(see \cite[Ch.9]{LovaszBook}).
More precisely, let $G$ be a finite simple graph on vertex set
$V(G) = [n]:= \{ 1, \ldots , n \}$ and edge set
$E(G) \subseteq {[n] \choose 2}$ and let $d\geq 1$ be an integer.
By $\overline{G}$ we denote the complementary graph of $G$ with edge set
$E({\overline{G}}) = {[n] \choose 2} \setminus E(G)$.
An orthogonal representation of $G$
in $\RR^d$ is a map $\varphi$ from $[n]$ to $\RR^d$ such that for any edge
$\{i,j\} \in E({\overline{G}})$ in the complementary graph, the vectors
$\varphi(i)$ and $\varphi(j)$ are orthogonal with respect to the standard scalar product in $\RR^d$.
Formulated differently, if we identify the image of the vertex $i$ with
the $i$-th row $(u_{i1},\ldots,u_{id})$ of an
$(n \times d)$-matrix $U = (u_{ij})_{(i,j) \in [n] \times [d]} \in \RR^{n \times d}$, then the set
of all orthogonal representations of the graph $G$ is the vanishing set
in $\RR^{n \times d}$ of the ideal $L_{\overline{G}} \subset
\RR[x_{ij}\:\; i=1,\ldots,n, \; j=1,\ldots,d]$, where $L_{\overline{G}}$ is
generated by the homogeneous polynomials
\begin{equation}
  \label{generators1}
  \sum_{k=1}^dx_{ik}x_{jk}
\end{equation}
for $\{i,j\} \in E(\overline{G})$. We write $\OR_d^{\RR}(G) \subseteq \RR^{n \times d}$ for the variety of
orthogonal representations of $G$.

The first study of $L_{\overline{G}}$ and the geometry
of $\OR_d^{\RR}(G)$ can be found in \cite{LSS}. For that reason we would
like to call the ideal $L_{\overline{G}}$ of orthogonal graph representations of $G$ the {\em Lov\'{a}sz-Saks-Schrijver ideal} of $G$.
Clearly, the variety $\OR_d^{\RR}(G)$ contains
many degenerate representations where, for example, one of the vertices is represented by
the zero vector. To avoid this kind of degeneracy,
Lov\'{a}sz, Saks and Schrijver in \cite{LSS} consider general-position
orthogonal representations, that is, orthogonal representations in which
any $d$ representing vectors are linearly independent.
In \cite[Thm. 1.1]{LSS} they prove the remarkable fact
that $G$ has such a representation in $\RR^d$ if and only if $G$ is
$(n-d)$-connected in which case $L_{\overline{G}}$ is a prime ideal.

While the results from \cite{LSS} express properties of $\OR_d^{\RR}(G)$ and
$L_{\overline{G}}$ in terms of graph-theoretic properties of $G$, it is sometimes
convenient to interchange the role of $G$ and $\overline{G}$. Therefore, in our
subsequent work we will refer to $L_G$ or $L_{\overline{G}}$ depending on
which point of view is more suitable for the given context.

In this paper we want to understand some algebraic properties of $L_{\overline{G}}$ and
the geometry of the variety $\OR_d^{K}(G)$ of orthogonal representations
of $G$ for general $G$ and over an arbitrary field $K$.
This appears to be a hard task and we confine ourselves
to the first interesting case, when $d = 2$ and $K$ an arbitrary field.
Note that, for $d=1$ the ideal $L_G$ is a monomial ideal known as the edge
ideal of $G$ and a well studied object (see for example \cite[Ch.~9]{HH}
or \cite{MV}).

For an easier reading in the case $d =2$ we rename the
two variables $x_{i1}$, $x_{i2}$ corresponding to the
coordinates of the $i$-th vertex as
$x_i$, $y_i$ and consider $L_G$ as an ideal in the ring
$K[x_1,\ldots, x_n, y_1, \ldots, y_n]$. For the rest of the paper we only deal with the case $d=2$, unless otherwise stated.

The most satisfying results are obtained
under the additional hypothesis that $\sqrt{-1} \not\in K$ or equivalently
when the bilinear form of the standard scalar product on $K^2$
is non-degenerate. In particular, all our main results apply in the
motivating situation $K = \RR$.

After providing preparatory results in Section~\ref{section1}
we formulate our first main results in Section~\ref{section0}.

\begin{Theorem}
  \label{radical-L_G}
  Let $G$ be a graph on $[n]$ and $\chara(K)\neq 2$.
  Then $L_{G}$ is a radical ideal.
  In particular, if $K$ is algebraically closed then $L_{\overline{G}}$ contains all
  polynomials vanishing on $\OR_2^{K}(G)$.
\end{Theorem}

\begin{Theorem}
  \label{char2}
  Let $G$ be a graph on $[n]$ and $\chara(K)=2$.
  Then $L_{G}$ is a radical ideal if and only if $G$ is bipartite.
\end{Theorem}

In order to prove \ref{radical-L_G} we first show that one may
assume that $\sqrt{-1}\in K$. Then we apply a linear transformation of
variables and obtain an ideal generated by the permanents of a
$(2 \times n)$-matrix whose column indices correspond to edges of $G$.
We call these ideals, which are of interest for their own sake,
permanental edge ideals. Permanental ideals have first been studied by Laubenbacher and Swanson in \cite{LS}.

In \ref{Grobner} we show that if $\chara(K)\neq 2$, then
permanental edge ideals have a squarefree initial ideal with respect
to a suitable monomial order.
>From this by a standard deformation argument it then follows that
permanental edge ideals and hence $L_G$ are radical ideals.

Based on experimental data we conjecture that for general $d$ the
following holds:

\begin{Conjecture}
  If $\chara(K) = 0$, then $L_G$ is a radical ideal for all $d$.
\end{Conjecture}

For $\sqrt{-1} \not\in K$, Section~\ref{section2} provides in
\ref{primarynew} a primary decomposition of $L_G$.
In geometric terms \ref{primarynew} has the following
reformulation.

\begin{Theorem}
  \label{geometriclovasz}
  If $\sqrt{-1} \not\in K$, then
  the variety $\OR_2^K(G)$ of orthogonal representations of a
  graph $G$ in $\RR^2$ is the union of varieties
  $V_S$ for $S \subseteq [n]$, each of them being the zero set of a prime ideal $Q_S$. Let $\overline{G}_1,\dots,\overline{G}_{c(S)}$ be the connected components of $\overline{G} \setminus S$. Then the variety $V_S$ contains all representations
  $\varphi : [n] \rightarrow \RR^2$ such that
  \begin{itemize}
    \item[(a)] $\varphi(i) = 0$ for $i \in S$ and all $i \in \overline{G}_j$, for which $\overline{G}_j$ is non-bipartite,
	\item[(b)] for each bipartite connected component $\overline{G}_\ell$ one has
    \begin{itemize}
    \item[(1)] $\varphi(i) \perp \varphi(j)$ if $i$ and $j$ are vertices of
	  the same bipartite connected component $\overline{G}_\ell$ and lie in
          different
	  blocks of its vertex bipartition,
	\item[(2)] $\varphi(i)$ and $\varphi(j)$ are linearly dependent if $i$ and $j$
	  are vertices of the same bipartite connected component $\overline{G}_\ell$ and
	  lie in the same block of its vertex bipartition.
    \end{itemize}
  \end{itemize}
  In particular, for each orthogonal representation in $V_S$ and each bipartite connected
  component $\overline{G}_\ell$ there is a pair of orthogonal lines in $\RR^2$ such that
  each of them contains the images from one block of the bipartition of $V(\overline{G}_\ell)$.
\end{Theorem}

Note, that the representation $\OR_2^K(G)$ as a union of $V_S$ is redundant
in general.

In Section \ref{section3} the primary decomposition of $L_G$ is turned into an
irredundant primary decomposition, see \ref{minimalprimes2}.
This theorem shows that the algebraic structure of the ideal $L_G$ is very
delicate and very sensitive toward the condition $\sqrt{-1} \not\in K$.
The formulation is slightly technical, but reveals the deeper correlation
of graph theoretical properties of $G$ and algebraic properties of $L_G$.

As a corollary one gets for $d=2$ the following generalization
of the result from \cite{LSS} to arbitrary fields $K$, where
$\sqrt{-1}$ is not in the prime field of $K$.

\begin{Corollary}
  \label{lovaszcor}
  Let $d=2$ and $\chara(K) = 0$ or $\chara(K) \not\equiv 1,2~\mod~4$. Then
  $L_{\overline{G}}$ is prime if and only if
  $G$ is $(n-2)$-connected.
\end{Corollary}

Using the formulation of $(n-2)$-connectedness of $G$ in terms of $\overline{G}$
from \ref{prime}, it is easy to see that for $(n-2)$-connected graphs $G$ the ideal
$L_{\overline{G}}$, and therefore also $\OR_2^K(G)$, is a complete
intersection. Indeed experimental data suggests that this is true in
general:

\begin{Conjecture}
  Let $d \geq 1$, $K$ be an arbitrary field and $G$ be a graph on vertex set $[n]$.
  If $G$ is $(n-d)$-connected, then $L_G$ is prime and a complete intersection.
\end{Conjecture}

In addition to our main results we supply results on other algebraic
properties of the ideals $L_G$, e.g. unmixedness, height of primary
components, etc..

We end this introduction by explaining the connection of $L_G$ and binomial
edge ideals. Note that this connection is only valid for $d=2$,
$G$ bipartite and $\sqrt{-1} \in K$. In particular, it does not hold for
$K = \RR$.

\begin{Remark}
\label{binomial edge ideal}
  If $\sqrt{-1}\in K$ and $G$ is a  bipartite graph, then
  $L_G$ may be identified with the binomial edge ideal of $G$ (see \cite{HHHKR}). Indeed, suppose  $V(G)=V_1\cup V_2$ is the bipartition of $G$
  with $|V_1|=m$ and $|V_2|=n$. We apply the automorphism of
  $K[x_1,\ldots, x_n,y_1,\ldots, y_n]$ to $L_G$ defined by
  $x_i\mapsto x_i$ and $y_i\mapsto \sqrt{-1}y_i$ to obtain
  the binomial edge ideal $J_G$ attached to the matrix
  \[
    \begin{bmatrix}
      z_1 & \cdots & z_n \\ w_1 & \cdots & w_n
    \end{bmatrix},
  \]
  where $z_i=x_i$ for $i=1,\ldots,m$, $z_i=\sqrt{-1}y_i$ for
  $i=m+1,\ldots,n$, $w_i=\sqrt{-1}y_i$ for $i=1,\ldots,m$, and
  $w_i=x_i$ for $i=m+1,\ldots,n$. In particular, by \cite{HHHKR},
  for a bipartite graph $G$, the primary decomposition as well
  as the initial ideal of $L_G$ is known with respect to the
  lexicographic order induced by $z_1>\cdots>z_n>w_1>\cdots>w_n$.
  Moreover, when $G$ is a bipartite graph it is known that
  $L_G$ is a radical ideal.
\end{Remark}

\section{The ideals $I_{K_{n}}$ and $I_{K_{m,n-m}}$}
\label{section1}

In this section we define and analyze the ideals $I_{K_n}$ and $I_{K_{m,n-m}}$
that will prove to be building blocks for the reduced primary decomposition of
$L_G$. In addition, on the way, we prove a result that will be needed when
studying the radicality of $L_G$.

Set $T = K[x_1,\ldots, x_n,y_1,\ldots,y_n]$ and consider the following two
ideals in $T$. The notation, which may look unwieldy in the first place
will allow for a succinct formulation of the primary decomposition.

\begin{itemize}
  \item[($I_{K_n}$)]
    We set $I_{K_1}=(0)$, $I_{K_2}=(x_1x_2+y_1y_2)$
	and for $n>2$, we define $I_{K_n}$ as the ideal generated by
    the binomials
    \begin{eqnarray}
      \label{generators2}
      f_{ij}&=& x_ix_j+y_iy_j, \quad 1\leq i<j\leq n,
      \nonumber
      \\
      g_{ij}&=&x_iy_j-x_jy_i, \quad 1\leq i<j\leq n,\\
      h_i&=&x_i^2+y_i^2, \quad\quad\quad  1\leq i\leq n.
      \nonumber
    \end{eqnarray}

  \item[($I_{K_{m,n-m}}$)]
   For $1\leq m<n$ we define $I_{K_{m,n-m}}$ as the ideal generated by
   the binomials

   \begin{eqnarray}
     \label{generators3}
     f_{ij}&=& x_ix_j+y_iy_j, \quad 1\leq i\leq m, \quad m+1\leq j\leq n,
     \nonumber
     \\
     g_{ij}&=&x_iy_j-x_jy_i, \quad 1\leq i<j\leq m \quad \text{or}
     \quad m+1\leq i<j\leq n.
  \end{eqnarray}
\end{itemize}

Throughout this paper, when we refer to the standard generators of the
ideals $L_G$, $I_{K_n}$ and $I_{K_{m,n-m}}$ we mean the generators
introduced in \eqref{generators1}, \eqref{generators2} and
\eqref{generators3}, respectively.

\begin{Lemma}
  \label{quadratic}
  The standard generators of $I_{K_n}$, and the standard generators of $I_{K_{m,n-m}}$
  form  a Gr\"obner basis with respect to the lexicographic order
  induced by $x_1>\cdots > x_n> y_1>\cdots > y_n$.
\end{Lemma}
\begin{proof}
  The assertion of the lemma follows once we have shown that
  for either of the ideals all $S$-polynomials of the standard generators
  reduce to zero, see for example \cite[Thm.~2.14]{EH} or
  \cite[Thm.~2.3.2]{HH}. If the
  leading terms of a pair of binomials do not have a common factor,
  then this $S$-polynomial reduces to zero, see for example
  \cite[Lem.~2.3.1]{HH}. Hence in the sequel we only have to consider the
  case that the leading terms have a common factor.
  In this case simple calculations show that such $S$-polynomials
  reduce to zero. We provide two examples and leave the remaining
  cases to the reader. First, for the standard generators $h_i$ and $f_{ij}$
  of $I_{K_n}$ we have $S(h_{i}, f_{ij})=-y_ig_{ij}$, and second for
  the standard generators $f_{ij}$ and $f_{ik}$ of $I_{K_{m,n-m}}$ we have
  $S(f_{ij}, f_{ik})=- y_ig_{jk}$ for $1\leq i\leq m$ and $m+1\leq j<k\leq n$.
\end{proof}

\begin{Corollary}
  \label{nonzerodivisor}
  Let $1 \leq m \leq n$.
  Then the variables  $x_1,\ldots,x_n, y_1,\ldots,y_{n}$ are non-zero
  divisors modulo $I_{K_n}$ and modulo $I_{K_{m,n-m}}$.
\end{Corollary}

\begin{proof}
  It follows from \ref{quadratic} that $y_1$ does not divide any of
  the monomial generators of $\ini_<(I_{K_n})$,
  where $<$ is the lexicographic order
  induced by $x_1>\cdots > x_{n}> y_1>\cdots > y_{n}$.
  This implies that
  $y_1$ is a non-zero divisor modulo $\ini_<(I_{K_n})$. Consequently,
  by \cite[Lem.~6.36]{EH}
  $y_1$ is a non-zero divisor modulo $I_{K_n}$.
  Using symmetry it follows that all $y_i$ are
  non-zero divisors
  modulo $I_{K_n}$. Furthermore, if we consider the initial ideal of
  $I_{K_n}$ with respect to the lexicographic order induced by
  $y_1>\cdots >y_n>x_1>\cdots>x_n$ then
  as before it follows that $x_1$ is a non-zero divisor modulo
  $\ini_<(I_{K_n})$,  and hence modulo $I_{K_n}$. Again by symmetry it follows
  that all $x_i$ are non-zero divisors modulo $I_{K_n}$.

  We apply \ref{quadratic} and deduce that $y_1$ and $y_{m+1}$ do not
  divide any of the monomial generators of $\ini_<(I_{K_{m,n-m}})$.
  This implies that $y_1$ and $y_{m+1}$ are non-zero divisors modulo
  $\ini_<(I_{K_{m,n-m}})$. Consequently,
  by \cite[Lem.~6.36]{EH} $y_1$ and $y_{m+1}$ are non-zero
  divisors modulo $I_{K_{m,n-m}}$.
  Again employing symmetry
  it follows that all $y_i$ are non-zero divisors modulo $I_{K_{m,n-m}}$.
  The same arguments as used for the $I_{K_n}$ now show that $x_1$ and
  $x_{m+1}$, and hence all $x_i$ are non-zero divisors modulo $I_{K_{m,n-m}}$.
\end{proof}

\begin{Proposition}
    \label{more}
    The ideals  $\ini_<(I_{K_n})$ and $\ini_<(I_{K_{m,n-m}})$ have a linear resolution, $\height(\ini_<( I_{K_n}))=n$,   $\height(\ini_<(I_{K_{m,n-m}}))=n-1$, $\depth (S/\ini_<(I_{K_n}))=1$ and  $S/\ini_<(I_{K_{m,n-m}})$ is Cohen--Macaulay. The same statements hold for the ideals $I_{K_n}$ and $I_{K_{m,n-m}}$.
\end{Proposition}

\begin{proof}
We first show that $\ini_<(I_{K_n})$ has a linear resolution. By Lemma~\ref{quadratic}, $J=\ini_<(I_{K_n})=J_1+J_2$, where $J_1=(x_1,\ldots,x_n)^2$ and
$J_2=(x_iy_j:1\leq i<j\leq n)$. We order the generators of this initial ideal in a way that the monomial generators of $J_1$ are bigger than the monomial
generators of $J_2$ and such that the generators of $J_1$ as well as the generators of $J_2$ are ordered lexicographically induced by $x_1>\cdots>x_n>y_1>\cdots>y_n$.
The ideal $J$ has linear quotients with respect to this ordering of its monomial generators. Indeed, the ideal $J_1$ is known to have linear quotients.
Now let $x_iy_j\in J_2$. We denote by $J_{ij}$ the ideal generated by all monomial generators of $J$ which are bigger than $x_iy_j$. Then
\[
J_{ij}:x_iy_j=(x_1,\ldots,x_n,y_{i+1},\ldots,y_{j-1}).
\]
This shows that $J$ has linear quotients and hence has a linear resolution by \cite[Proposition~8.2.1]{HH}. Moreover it follows from \cite[Corollary~8.2.2]{HH} that $\projdim (J)=2n-2$, because
$J_{1n}:x_1y_n=(x_1,\ldots,x_n,y_2,\ldots,y_{n-1})$ and any other colon ideal has less generators. This implies that $\depth (S/J)=1$.
Furthermore, one can see that the ideal $(x_1,\ldots,x_n)$ whose height is $n$ is a minimal prime ideal of $\ini_<(I_{K_{n}})$ and all other minimal
primes of $\ini_<(I_{K_{n}})$ have height greater than $n$. Thus $\height (\ini_<(I_{K_{n}}))=n$.
Now we show that $\ini_<(I_{K_{m,n-m}})$ has a linear resolution. By Lemma~\ref{quadratic},
\begin{eqnarray}
\ini_<(I_{K_{m,n-m}})&=&(x_ix_j:1\leq i\leq m, m+1\leq j\leq n)
\label{initial} \\
&+&(x_iy_j:1\leq i<j\leq m~\text{or}~m+1\leq i<j\leq n),
\nonumber
\end{eqnarray}
which is the edge ideal of a connected bipartite graph $H$ on the vertex set $V=V_1\cup V_2$ with $V_1=\{x_1,\ldots,x_m,y_{m+2},\ldots,y_n\}$ and $V_2=\{x_{m+1},\ldots,x_n,y_2,\ldots,y_m\}$. We label the vertices of $H$ such that $V_1=\{v_1,\ldots,v_{n-1}\}$ and $V_2=\{w_1,\ldots,w_{n-1}\}$, where $v_i=x_i$ for $1\leq i\leq m$, $v_i=y_{n+m+1-i}$ for $m+1\leq i\leq n-1$, $w_i=x_{m+i}$ for $1\leq i\leq n-m$, and $w_i=y_{n-i+1}$ for $n-m+1\leq i\leq n-1$. Figure~\ref{Ferrers} shows a simple example of such a graph for
$m=2$ and $n=5$. We show that $H$ is a Ferrers bipartite graph (by this labeling).
Recall that a Ferrers graph is a bipartite graph $H'$ on $V=A\cup B$ with $A=\{a_1,\ldots,a_p\}$ and $B=\{b_1,\ldots,b_q\}$ such that $\{a_1,b_q\}\in E(H')$, $\{a_p,b_1\}\in E(H')$, and if $\{a_i,b_j\}\in E(H')$, then $\{a_t,b_l\}\in E(H')$ for all $1\leq t\leq i$ and $1\leq l\leq j$. Associated to a Ferrers graph $H'$ is a sequence $\lambda=(\lambda_1,\ldots,\lambda_p)$ of non-negative integers, where $\lambda_i=\deg_{H'} a_i$ which is the degree of the vertex $a_i$ in $H'$ for all $i=1,\ldots,p$.
The edges of $H$ are given by (\ref{initial}), and it can be seen that $H$ with the labeling of the vertices as given above, is a Ferrers graph. Therefore, by \cite[Theorem~2.1]{CN} $\ini_<(I_{K_{m,n-m}})$ has a linear resolution.

\begin{figure}[hbt]
\begin{center}
\psset{unit=0.8cm}
\begin{pspicture}(-1,-2)(4,2)
\rput(-0.8,-1){
\rput(0,0){$\bullet$}
\rput(1.5,0){$\bullet$}
\rput(3,0){$\bullet$}
\rput(4.5,0){$\bullet$}

\rput(0,2){$\bullet$}
\rput(1.5,2){$\bullet$}
\rput(3,2){$\bullet$}
\rput(4.5,2){$\bullet$}
\psline(0,0)(0,2)
\psline(0,0)(1.5,2)
\psline(0,0)(3,2)
\psline(0,0)(4.5,2)
\psline(1.5,0)(0,2)
\psline(1.5,0)(1.5,2)
\psline(1.5,0)(3,2)
\psline(3,0)(0,2)
\psline(3,0)(1.5,2)
\psline(4.5,0)(0,2)

\rput(0,-0.5){$x_{3}$}
\rput(1.5,-0.5){$x_{4}$}
\rput(3,-0.5){$x_{5}$}
\rput(4.5,-0.5){$y_{2}$}

\rput(0,2.5){$x_{1}$}
\rput(1.5,2.5){$x_{2}$}
\rput(3,2.5){$y_{5}$}
\rput(4.5,2.5){$y_{4}$}
}
\end{pspicture}
\end{center}
\caption{}\label{Ferrers}
\end{figure}
Now we show that $S/\ini_<(I_{K_{m,n-m}})$ is Cohen--Macaulay. For this purpose, by \cite[Corollary~2.7]{CN}, we need  to compute the sequence $\lambda$ associated to $H$. By (\ref{initial}), $\deg_{H} v_i=\deg_{H} x_i= n-i$ for all $i=1,\ldots,m$. Moreover, since by (\ref{initial}), $\deg_{H} y_j=j-1-m$, for all $j=m+2,\ldots,n$, it follows that $\deg_{H} v_i=\deg_{H} y_{n+m+1-i}=n-i$ for all $i=m+1,\ldots,n-1$. Therefore, $\lambda=(n-1,n-2,\ldots,2,1)$ is the associated sequence to the Ferrers graph $H$, and hence by \cite[Corollary~2.7]{CN}, it follows that $S/\ini_<(I_{K_{m,n-m}})$ is Cohen--Macaulay. In particular,   $\ini_<(I_{K_{m,n-m}})$ is an unmixed ideal of height $|V_1|=|V_2|=n-1$, since $V_1$ and $V_2$ are minimal vertex covers of $H$.
By \cite[Corollary~22.13]{P}, $I_{K_{n}}$ and $\ini_<(I_{K_{n}})$ have the same minimal graded free resolution. It follows that $I_{K_{n}}$ has a linear
resolution and $\projdim (I_{K_{n}})=\projdim (\ini_<(I_{K_{n}}))$, and hence $\depth (S/I_{K_{n}})=\depth (S/\ini_<(I_{K_{n}}))=1$. Similarly by
\cite[Corollary~22.13]{P}, $I_{K_{m,n-m}}$ has a linear resolution. By \cite[Theorem~3.3.4,~(a)]{HH} we have $\height (I_{K_{n}})=\height (\ini_<(I_{K_{n}}))=n$ and $\height (I_{K_{m,n-m}})=\height (\ini_<(I_{K_{m,n-m}}))=n-1$. Finally, by \cite[Corollary~3.3.5]{HH} $S/I_{K_{m,n-m}}$ is also Cohen--Macaulay.
\end{proof}

\begin{Theorem}
  \label{primeness-K_{m,n-m}}
  The ideal $I_{K_{m,n-m}}$ is a prime ideal.
\end{Theorem}

In the proof of \ref{primeness-K_{m,n-m}} we will study the
localization $T_Y$ of $T$ at the multiplicative set
$Y = \{ (y_1\cdots y_n)^i~:~i \geq 1\}$ of all powers of
$y_1\cdots y_n$. Since we will make use of this ring also subsequently
we fix this notation for the rest of the paper.
Notice that
\[
  T_Y=K[z_1,\ldots,z_{n},y_1^{\pm 1},\ldots, y_{n}^{\pm 1}],
\]
with $z_i=x_i/y_i$ for $i=1,\ldots,n$, and furthermore, that the elements
$z_1,\ldots,z_{n}$ are algebraically independent.

\begin{proof}
  Because of \ref{nonzerodivisor} it suffices to show that $I_{K_{m,n-m}}T_Y$ is
  a prime ideal in the ring $T_Y$.
  Since in $T_Y$ all $y_i$ are units, the generators \eqref{generators3} of $I_{K_{m,n-m}}T_Y$ can be
  rewritten as
  \begin{eqnarray}
    z_iz_j+1, & 1\leq i\leq m, \quad m+1\leq j\leq n,\label{othergenerators}\\
    z_i-z_j, & 1\leq i<j\leq m \quad \text{or} \quad m+1\leq i<j\leq n.\label{linearforms}
  \end{eqnarray}
  In order to see that $T_Y/I_{K_{m,n-m}}T_Y$ is a domain, we first consider the
  quotient $R$ of $T_Y$ by the linear forms in \eqref{linearforms} and denote
  by $\overline{I}$ the image of $I_{K_{m,n-m}}T_Y$ in $R$.
  Notice that $R$ is isomorphic to
  $K[z_1,z_{m+1}, y_1^{\pm 1},\ldots, y_{n}^{\pm 1}]$, and that
  $T_Y/I_{K_{m,n-m}}T_Y \iso R/ \overline{I}R$. Since the residue class map
  $T_Y\to R$ identifies  $z_i$ with $z_1$ for $i=1,\ldots,m$  and with $z_{m+1}$
  for $i=m+1,\ldots,n$,  we see that $\overline{I}=(z_1z_{m+1}+1)$. Since the
  polynomial $z_1z_{m+1}+1$ is irreducible, we conclude that $R/ \overline{I}R$,
  and hence also $T_Y/I_{K_{m,n-m}}T_Y$ is a domain, as desired.
\end{proof}

\begin{Theorem}
  \label{primeness-K_{n}}
  Let $n>2$ be an integer.
  \begin{enumerate}
  \item[{\rm (a)}] If $\sqrt{-1}\notin K$, then $I_{K_n}$ is a prime ideal.
  \item[{\rm (b)}] If $\sqrt{-1}\in K$ and $\chara(K)\neq 2$, then $I_{K_n}$ is
    a radical ideal. More precisely,
    \[
      I_{K_n}=(x_1+\sqrt{-1}y_1,\ldots,x_n+\sqrt{-1}y_n)\sect (x_1-\sqrt{-1}y_1,\ldots,x_n-\sqrt{-1}y_n).
    \]
  \item[{\rm (c)}] If $\chara(K)=2$, then $I_{K_n}$ is a primary  ideal with
    \[
      \sqrt{I_{K_n}}=(x_1+y_1,\ldots,x_n+y_n).
    \]
  \end{enumerate}
\end{Theorem}

\begin{proof}
   As in the proof of \ref{primeness-K_{m,n-m}}
   we consider the image $I_{K_n}T_Y$ of our ideal in
   $T_Y = K[z_1,\dots, z_n,y_1^{\pm 1},\ldots, y_n^{\pm 1}]$.
   It is generated by the polynomials
   \begin{eqnarray}
     \label{localized}
     z_iz_j+1, z_i-z_j,\quad 1\leq i<j\leq n,\label{alsolinear} \\
     z_i^2+1, \quad\quad\quad  1\leq i\leq n.\label{squaregenerators}
   \end{eqnarray}
   Let $R$ be the residue class ring of $T_Y$ modulo the linear forms
   given in \eqref{alsolinear}. Then
   $R\iso K[z_1,y_1^{\pm 1},\ldots, y_{n}^{\pm 1}]$
   and $T_Y/I_{K_n}T_Y \iso R/(z_1^2+1)$.
   \begin{itemize}
     \item[(a)]
       It follows that if $\sqrt{-1}\not\in K$,
       then $T_Y/I_{K_n}T_Y$ is a domain, and hence $I_{K_n}$ is a prime ideal.
     \item[(b)] Since $T_Y/I_{K_n}T_Y\iso R/((z_1+\sqrt{-1})(z_1-\sqrt{-1}))$
       it follows that $I_{K_n}$ is radical and has exactly two minimal prime
       ideals. The ideals
       $P_1=(x_1+\sqrt{-1}y_1,\ldots,x_n+\sqrt{-1}y_n)$ and
       $P_2=(x_1-\sqrt{-1}y_1,\ldots,x_n-\sqrt{-1}y_n)$ are prime ideals
       of height $n$ containing $I_{K_n}$. By
       \ref{more} we know that $\height (I_{K_n})=n$. It follows that
       $\{P_1, P_2\}$ is the set of  minimal prime ideals of $I_{K_n}$.
     \item[(c)] Since $\chara(K)=2$, we have $x_i^2+y_i^2=(x_i+y_i)^2$
       for all $i$. This shows that $I_{K_n}$ is not a prime ideal in this case.
       Furthermore, it follows that $x_i+y_i\in \sqrt{I_{K_n}}$ for all $i$.
       Since for all $i<j$, $g_{ij}=(x_i+y_i)x_j+(x_j+y_j)x_i$ and
       $f_{ij}=(x_i+y_i)x_j+(x_j+y_j)y_i$, we see that
       $\sqrt{I_{K_n}}=(x_1+y_1,\ldots, x_n+y_n)$, as desired.
  \end{itemize}
\end{proof}

\section{A Gr\"{o}bner basis for permanental edge ideals}
\label{section0}

The main goal of this section is to prove \ref{radical-L_G}
and \ref{char2}. For the first result we apply a
certain linear transformation of coordinates, and show that $L_G$ admits
a squarefree initial ideal with respect to the new coordinates and a
suitable monomial order. Assume that $\sqrt{-1}\in K$
and $\chara(K)\neq 2$.
We consider the following linear transformation $\varphi$ with
$\varphi(x_i)=x_i - y_i$ and $\varphi(y_i)= \sqrt{-1} (x_i + y_i)$ for
all $i$. Then for every $i \neq j$, the binomial $x_i x_j + y_i y_j$ maps
to $-2(x_i y_j + x_j y_i)$. Let $\Pi_{G}=\varphi(L_G)$. Since
$\chara(K) \neq 2$, it follows that
\[
  \Pi_G = (x_i y_j + x_j y_i\, :\, \{i,j\} \in E(\overline{G})).
\]
The generators of $\Pi_G$ are those $2$-permanents of the matrix
$
  \begin{bmatrix}
    x_{1} & \cdots & x_{n} \\ y_{1} & \cdots & y_{n}
  \end{bmatrix}
$
whose column indices correspond to edges of $\overline{G}$. Therefore, we call
$\Pi_G$ the \textit{permanental edge ideal} of $\overline{G}$.
Replacing the above $2$-permanents by $2$-minors, one obtains the classical
binomial edge ideal of $\overline{G}$.

In order to describe the Gr\"{o}bner basis of the ideal $\Pi_G$ we
introduce some terminology and notation. Let $i$ and $j$ be
two distinct vertices of $G$. A path of length $r$ in a graph $G$ on
vertex set $[n]$ from $i$ to $j$ is a sequence
$\pi_{ij}:i=i_0,i_1,\dots,i_r=j$ of pairwise distinct
vertices such that $\{i_k,i_{k+1}\}\in E(G)$ for all $k$. By an even
(odd) path we mean a path of even (odd) length. We say that the path
$\pi_{ij}$ is \textit{admissible} if $i<j$, and for each $k=1,\dots,r-1$,
one has either $i_k < i$ or $i_k > j$. In the case that $\pi_{ij}$ is
admissible we attach to it the monomial
\[
  u_{\pi_{ij}} = \prod_{i_k>j} x_{i_k} \prod_{i_k<i} y_{i_k}.
\]

As a preparation for the description of a Gr\"obner basis of
$\Pi_G$ we prove the following lemma.

\begin{Lemma}
  \label{lemmapaths}
  Let $G$ be a graph on vertex set $[n]$.
  Let $k < \ell$ and let $\pi_{k\ell}$ and $\sigma_{k\ell}$ be two admissible paths from
  $k$ to $\ell$ in $G$ such that
  \begin{itemize}
     \item $\pi_{k\ell}$ is odd and $\sigma_{k\ell}$ is even,
     \item $D=\{ j \in V(\pi_{k\ell}) \cup V(\sigma_{k\ell}) : j>\ell \}
	      \neq \emptyset$.
  \end{itemize}
  Set $h = \min D$.
  Then there are two admissible paths
  $\tau_1$ and $\tau_2$ both with endpoints either $k$ and $h$ or
  $\ell$ and $h$ such that one of $\tau_1$ and $\tau_2$ is odd and
  one is even. In particular, the monomial
  \[
    \prod_{\substack{j \leq \ell,\\ j \in V(\tau_1) \cup V(\tau_2)}} y_j
    \prod_{\substack{j \geq h,\\ j \in V(\tau_1) \cup V(\tau_2)}} x_j \quad \text{divides }
    \prod_{\substack{j \leq \ell,\\ j \in V(\pi_{k\ell}) \cup V(\sigma_{k\ell})}} y_j
    \prod_{\substack{j \geq h,\\ j \in V(\pi_{k\ell}) \cup V(\sigma_{k\ell})}} x_j.
  \]
\end{Lemma}

\begin{proof}
  Let $\pi_{k\ell}: k=i_0,i_1,\dots,i_r=\ell$ and
  $\sigma_{k\ell}: k=i'_0,i'_1,\dots,i'_s=\ell$.

  \smallskip

  \noindent {\sf Case 1:} $h \in V(\pi_{k\ell}) \cap V(\sigma_{k\ell})$.

  Then $h$ splits $\pi_{k\ell}$ and $\sigma_{k\ell}$ in two admissible paths:
  \[
    \pi_{k\ell} = \pi_{kh} \cup \pi_{\ell h} \text{ and }
    \sigma_{k\ell} = \sigma_{kh} \cup \sigma_{\ell h}.
  \]

  Since $\pi_{k\ell}$ is odd, it follows that $\pi_{kh}$ is odd and
  $\pi_{\ell h}$ is even or vice versa. Suppose that $\pi_{kh}$ is odd and
  $\pi_{\ell h}$ is even. Similarly, since $\sigma_{k\ell}$ is even, it
  follows that $\sigma_{kh}$ and $\sigma_{\ell h}$ are both odd or both even.
  If they are both odd we set $\tau_1 = \pi_{\ell h}$ and
  $\tau_2 = \sigma_{\ell h}$, otherwise we set $\tau_1 = \pi_{kh}$ and
  $\tau_2 = \sigma_{kh}$. Notice that $\tau_1$ and $\tau_2$ are admissible
  in both cases.

  \smallskip

  \noindent {\sf Case 2:} $h$ belongs to exactly one of the paths $\pi_{k\ell}$
  and $\sigma_{k\ell}$.

  Without loss of generality, we may assume that
  $h = i_c \in V(\pi_{k\ell}) \setminus V(\sigma_{k\ell})$ for some
  $c \in \{ 1,\dots,r-1 \}$. Then $h$ belongs to a cycle of
  $\pi_{k\ell} \cup \sigma_{k\ell}$. Notice that
  $\pi_{k\ell} \cup \sigma_{k\ell}$ contains at least one odd cycle $C$,
  since $\pi_{k\ell}$ is odd and $\sigma_{k\ell}$ is even. Let $a,b$ be the
  minimum integers such that $i_a=i'_b \in V(C)$ and let $d,e$ be the
  maximum integers such that $i_d=i'_e \in V(C)$. Let $h \notin V(C)$.
  If $i_d < i_c=h$, then we define
  \[
    \tau_1: k=i_0,i_1,\dots,i_c=h \text{ and }
  \]
  \[
    \tau_2: k=i_0,i_1,\dots,i_a=i'_b,i'_{b+1},\dots,i'_e=i_d,i_{d+1},\dots,i_c=h,
  \]
  where $\tau_1$ and $\tau_2$ are admissible paths, $\tau_1$ is odd and
  $\tau_2$ is even or vice versa, since $C$ is an odd cycle. On the other
  hand, if $i_a > i_c=h$, then we define
  \[
    \tau_1: \ell=i_r,i_{r-1},\dots,i_c=h \text{ and }
  \]
  \[
    \tau_2: \ell=i_r,i_{r-1},\dots,i_d=i'_e,i'_{e-1},\dots,i'_b=i_a,i_{a-1},\dots,i_c=h,
  \]
  where $\tau_1$ and $\tau_2$ are admissible paths, $\tau_1$ is odd and
  $\tau_2$ is even or vice versa, since $C$ is an odd cycle.

  Now suppose that $h \in V(C)$. We define
  \[
    \tau_1: k=i_0,i_1,\dots,i_c=h \text{ and }
  \]
  \[
    \tau_2: k=i_0,i_1,\dots,i_a=i'_b,i'_{b+1},\dots,i'_e=i_d,i_{d-1},\dots,i_c=h,
  \]
  where $\tau_1$ and $\tau_2$ are admissible paths, $\tau_1$ is odd and
  $\tau_2$ is even or vice versa, since $C$ is an odd cycle.

  In each case the last part of the statement follows.
\end{proof}

\begin{Theorem}
  \label{Grobner}
  Let  $G$ be a graph on $[n]$ and assume that $\chara(K)\neq 2$. Then with
  respect to the lexicographic order on $K[x_1,\ldots, x_n,y_1,\ldots,y_n]$
  induced by $x_1 > \cdots > x_n > y_1 > \cdots > y_n$ the following elements
  form a Gr\"{o}bner basis of the ideal $\Pi_G$:
  \begin{itemize}
    \item[{\rm (i)}] $u_{\pi_{ij}}b_{ij}$, where $\pi_{ij}$ is an odd admissible
      path and $b_{ij} = x_i y_j + x_j y_i$,
    \item[{\rm (ii)}] $u_{\pi_{ij}}g_{ij}$, where $\pi_{ij}$ is an even admissible
      path and $g_{ij} = x_i y_j - x_j y_i$,
    \item[{\rm (iii)}] $\lcm(u_{\pi_{ij}}, u_{\sigma_{ij}})y_ix_j$, where
      $\pi_{ij}$ is an odd and $\sigma_{ij}$ is an even admissible path,
    \item[{\rm (iv)}]
      $
       \begin{cases}
        y_b \prod_{h \in W} x_h & \text{if $b<h$ for every $h \in W$}\\
        x_b \prod_{h \in W} y_h & \text{if $b>h$ for every $h \in W$}
       \end{cases}
      $,

      where $W = V(\pi_{ij}) \cup V(\sigma_{ij}) \cup V(\tau_{ab}) \setminus
      \{b\}$, $\pi_{ij}$ is an odd and $\sigma_{ij}$ is an even admissible
      path from $i$ to $j$, $\tau_{ab}$ is a path with endpoints $a$ and $b$,
      such that $a$ is the only vertex of $\tau_{ab}$ that belongs to
      $V(\pi_{ij}) \cup V(\sigma_{ij})$.
  \end{itemize}
\end{Theorem}

In some parts the following proof is similar to the proof
of corresponding statement for binomial edge ideals
(see \cite{HHHKR}),
but one of the substantial differences is that the Gr\"obner
basis not only contains binomials but also monomials.

\begin{proof}
  Let $\MG$ be the set of elements listed in (i),(ii), (iii) and (iv).

  \smallskip

  \noindent {\sf Claim 1:} $\MG\subset \Pi_G$.

  \indent Note that $b_{ij}\in \MG$ for any edge $\{i,j\}$ of $G$, since each edge is clearly
  an admissible path, so that the generators of $\Pi_G$ belong to $\MG$.
  Now let $i,j \in [n]$, with $i<j$, and let
  $\pi_{ij}: i=i_0,i_1,\dots,i_r=j$ be an admissible path in $G$ from $i$
  to $j$. We show, by induction on $r \geq 1$, that
  $u_{\pi_{ij}} b_{ij} \in \Pi_G$, if $r$ is odd, and
  $u_{\pi_{ij}} g_{ij} \in \Pi_G$, if $r$ is even. The assertion is clearly
  true for $r=1$. Let $r=2$, if $i_1<i$, then
  $S(b_{i_1 i},b_{i_1 j}) = u_{\pi_{ij}}g_{ij}$. If $i_1>j$, then
  $S(b_{i i_1},b_{j i_1}) = -u_{\pi_{ij}}g_{ij}$. Now suppose $r>2$ and define
  the sets $A = \{i_k\, :\, i_k < i\}$ and $B = \{i_\ell\, :\, i_\ell > j\}$.
  One has either $A \neq \emptyset$ or $B \neq \emptyset$. If
  $A \neq \emptyset$, let $i_{k_0} = \max A$, while  if $B \neq \emptyset$,
  let $i_{\ell_0} = \min B$.
  Suppose that $A \neq \emptyset$. It then follows that both paths
  $\pi_1: i_{k_0},i_{k_0-1},\dots,i_1,i_0=i$ and
  $\pi_2: i_{k_0},i_{k_0+1},\dots,i_{r-1},i_r=j$ in $G$ are admissible.
  Now, if $r$ is odd, the induction hypothesis guarantees that either
  $u_{\pi_1} b_{i_{k_0},i}, u_{\pi_2} g_{i_{k_0},j} \in \Pi_G$ or
  $u_{\pi_1} g_{i_{k_0},i}, u_{\pi_2} b_{i_{k_0},j} \in \Pi_G$. In the first
  case, we observe that the $S$-polynomial
  $S(u_{\pi_1} b_{i_{k_0},i}, u_{\pi_2} g_{i_{k_0},j}) =
    u_{\pi_1} u_{\pi_2} y_j (x_{i_{k_0}} y_i + x_i y_{i_{k_0}}) -
    u_{\pi_1} u_{\pi_2} y_i (x_{i_{k_0}} y_j - x_j y_{i_{k_0}}) =
    u_{\pi_{ij}} b_{ij}$,
  hence $u_{\pi_{ij}} b_{ij} \in \Pi_G$. The same conclusion holds in the
  second case. On the other hand, if $r$ is even, the induction hypothesis
  guarantees that either $u_{\pi_1} b_{i_{k_0},i}, u_{\pi_2} b_{i_{k_0},j}
  \in \Pi_G$ or $u_{\pi_1} g_{i_{k_0},i}, u_{\pi_2} g_{i_{k_0},j} \in \Pi_G$.
  In the first case, it is easy to check that
  $S(u_{\pi_1} b_{i_{k_0},i}, u_{\pi_2} b_{i_{k_0},j}) = u_{\pi_{ij}} g_{ij}$,
  hence $u_{\pi_{ij}} g_{ij} \in \Pi_G$. The same conclusion holds also in the
  second case. Similarly one treats the case when $B \neq \emptyset$.

  Moreover, if $\pi_{ij}$ is an odd admissible path and $\sigma_{ij}$ is an
  even admissible path, then $S(u_{\pi_{ij}} b_{ij}, u_{\sigma_{ij}} g_{ij})=
  2 \lcm(u_{\pi_{ij}}, u_{\sigma_{ij}})y_ix_j$. Hence the monomials in (iii) belong to $\Pi_G$.

  Finally, we show that the monomials in (iv) belong to $\Pi_G$. In order to simplify the notation we
  set
  \[
    u_{\pi_{ij}}h_{ij}=\left\{
    \begin{array}{ll}
	u_{\pi_{ij}}b_{ij} & \text{ if $\pi_{ij}$ is an odd path}\\
	u_{\pi_{ij}}g_{ij} & \text{ if  $\pi_{ij}$ is an even path}
    \end{array}\right. .
  \]

  Let us consider $x_b \prod_{h \in W} y_h$ and an odd cycle $C$ contained in
  $\pi_{ij} \cup \sigma_{ij}$. Call $\ell$ the maximum vertex of this
  cycle, consider a path $\tau$ from $\ell$ to $b$ and relabel its vertices as
  $\tau: \ell=j_0,j_1,\dots,j_t=b$. Define $j_{t(0)}=j_0$, $j_{t(1)} =
  \min \{ j_c \in V(\tau) : j_c>\ell, c=1,\dots,t \}$, $j_{t(2)} =
  \min \{ j_c \in V(\tau) : j_c>j_{t(1)}, c=t(1)+1,\dots,t \}$. Proceeding
  in this way we find the sequence $\ell=j_{t(0)}<j_{t(1)}<\dots<j_{t(q)}=b$
  and for each $1 \leq c \leq q$, the path $\tau_c: j_{t(c-1)},j_{t(c-1)+1},
  \dots,j_{t(c)}$ is admissible. Then
  \begin{equation}
    \label{fourthmonomial}
    x_b \prod_{h \in W} y_h = (-1)^{t(q)-t(0)}
    \frac{x_b \prod_{h \in W} y_h}{x_b y_\ell} y_b x_\ell +
    \sum_{c=1}^q (-1)^{t(q)-t(c-1)+1} v_{\tau_c} u_{\tau_c}
    h_{j_{t(c-1)} j_{t(c)}},
  \end{equation}
  where $v_{\tau_c} = \frac{x_b \prod_{h \in W} y_h}{u_{\tau_c} y_{j_{t(c-1)}}
  y_{j_{t(c)}} x_b} y_b$ and $u_{\tau_c} h_{j_{t(c-1)} j_{t(c)}} \in
  \mathcal G$. We show that the first summand in \eqref{fourthmonomial} is a
  multiple of a monomial of the form (iii). Let
  $k = \max \{ h \in V(C) : h<\ell \}$, then we have two admissible paths
  $\pi_{k\ell}$ and $\sigma_{k\ell}$ from $k$ to $\ell$ whose union is the cycle $C$
  and such that one of them is odd and the other one is even. Suppose that
  $\pi_{k\ell}$ is odd, then $\frac{\prod_{h \in W} y_h}{y_\ell} y_bx_\ell = x_\ell
  \prod_{h \in C \setminus \{\ell\}} y_h =
  \lcm(u_{\pi_{k\ell}},u_{\sigma_{k\ell}}) y_k x_\ell$ is a monomial of the form (iii),
  hence it belongs to $\mathcal G$, and divides the last summand in
  \eqref{fourthmonomial}.

  \smallskip

  \noindent {\sf Claim 2:} All $S$-polynomials of elements of
  $\MG$ reduce to zero.

  \indent We distinguish several cases of $S$-polynomials for the binomials and
  monomials from (i)-(iv).

  \noindent {\sf Case I:} $S$-polynomials of binomials from (i) and (ii)

  Notice that
  $S(u_{\pi_{ij}}h_{ij}, u_{\sigma_{k\ell}}h_{k\ell})$ reduces to zero,
  if $\{i,j\}\sect\{k,\ell\}=\emptyset$ or if $i=\ell$, or $k=j$, because
  in these cases  the initial monomials $\ini_<(h_{ij})$ and
  $\ini_<(h_{k\ell})$ form a regular sequence, so that the assertion follows
  from \cite[Prob. 2.17]{EH}. Also
  $S(u_{\pi_{ij}}h_{ij}, u_{\sigma_{k\ell}}h_{k\ell}) = 0$, if $i=k, j=\ell$
  and $\pi_{ij}, \sigma_{ij}$ are both odd or both even. Thus there remain
  the cases that $i=k$ and $j \neq \ell$, or $j=\ell$ and $i \neq k$, and
  the case that $i=k$, $j=\ell$, the path $\pi_{ij}$ is odd and the path
  $\sigma_{ij}$ is even, or vice versa. In the last case the $S$-polynomial
  yields a scalar multiple of one of the monomials in (iii). Thus we need only
  to deal with the first two cases. For that we may assume that $i=k$ and $j<\ell$.

  Let $\pi_{ij}: i=i_0,i_1,\ldots,i_r=j$ and $\sigma_{i\ell}:
  i=i_0',i_1',\ldots,i_s'=\ell$. Then there exist indices $a$ and $b$ such
  that
  \[
    i_a=i_b'\quad \text{and}\quad
    \{i_{a+1},\ldots,i_r\}\sect\{i_{b+1}',\ldots,i_s'\}=\emptyset.
  \]

  Consider the path
  \[
    \tau: j=i_r,i_{r-1},\ldots,i_{a+1},i_a=i'_b,i'_{b+1},\ldots,i'_{s-1}, i'_s=\ell
  \]
  from $j$ to $\ell$. To simplify the notation we write this path as
  $\tau\: j=j_0,j_1,\ldots,j_t=\ell$. Notice that
  \begin{equation}\label{Spairs}
    S(u_{\pi_{ij}} h_{ij},u_{\sigma_{i\ell}} h_{i\ell}) =
    \begin{cases}
      w g_{j\ell}, & \text{if $\pi_{ij}$ and $\sigma_{i\ell}$ are both odd
      or both even}, \\
      w b_{j\ell}, & \text{otherwise},
    \end{cases}
  \end{equation}
  where $w = (-1)^{r+1} y_i \lcm(u_{\pi_{ij}},u_{\sigma_{i\ell}})$.

  Suppose that we are in the first case of \eqref{Spairs}. If the path
  $\tau$ from $j$ to $\ell$ is even, then we proceed as in the proof of
  \cite[Thm. 2.1]{HHHKR}, and set
  $j_{t(1)} = \min\{ \, j_c \, : \, j_c > j, \, c = 1, \ldots, t \, \}$ and
  $j_{t(2)} = \min\{ \, j_c \, : \, j_c > j, \, c = t(1) + 1, \ldots, t \, \}$.
  Proceeding in this way yields the integers
  $0 = t(0) < t(1) < \cdots < t(q) = t$. It then follows that
  $j =j _{t(0)} <  j_{t(1)} < \cdots < j_{t(q)} = \ell$
  and for each $1 \leq c \leq t$, the path
  $\tau_c : j_{t(c-1)}, j_{t(c-1)+1}, \ldots, j_{t(c)-1}, j_{t(c)}$
  is admissible.

  As in the proof of \cite[Thm. 2.1]{HHHKR} one shows that
  \[
    S(u_{\pi_{ij}} h_{ij}, u_{\sigma_{i\ell}} h_{i\ell}) =
    \sum_{c=1}^{q} (-1)^{t(c-1)}v_{\tau_c}
    u_{\tau_c}h_{j_{t(c-1)} j_{t(c)}}
  \]
  is a standard expression of
  $S(u_{\pi_{ij}} h_{ij}, u_{\sigma_{i\ell}} h_{i\ell})$ whose remainder is
  equal to $0$, where each $v_{\tau_c}$ is the monomial defined as
  \[
    v_{\tau_c}=\frac{w x_j x_\ell}{u_{\tau_c} x_{j_{t(c-1)}}x_{j_{t(c)}}},
    \quad\quad\text{for $1 \leq c \leq q$}.
  \]

  On the other hand, if the path $\tau$ is odd, then we need a different
  argument. First observe that one of the paths $\pi_1 = \pi_{ij}$ and
  $\pi_2 : i=i'_0,\dots,i'_b=i_a, i_{a+1},\dots, i_r=j$ is even and the other
  one is odd. Moreover they are both admissible: this is clear for
  $\pi_1$; as for $\pi_2$, if $h \in \{i'_0,\dots,i'_b\}$, then $h \leq i$ or
  $h > \ell > j$, if $h \in \{i_{a+1},\dots,i_r\}$, then $h<i$ or $h \geq j$.
  Since $\pi_1$ and $\pi_2$ have the same endpoints, they produce the
  monomial $\lcm(u_{\pi_1}, u_{\pi_2}) y_ix_j$ of the form (iii) that divides
  $wx_jy_\ell$.

  Now let $D = \{h \in V(\pi_1) \cup V(\pi_2) : h>j\}$. If $D = \emptyset$,
  then define the paths
  $\pi_3 : i=i_0,\dots,i_a=i'_b, i'_{b+1},\dots,i'_s=\ell$ and
  $\pi_4 = \sigma_{i\ell}$. One of them is odd and the other one is even and
  both are admissible: this is clear for $\pi_4$; as for $\pi_3$, if
  $h \in \{i_0,\dots,i_a\}$, then $h \leq i$ because $D=\emptyset$, if
  $h \in \{i'_{b+1},\dots,i'_s\}$, then $h < i$ or $h \geq \ell$. Since
  $\pi_3$ and $\pi_4$ have the same endpoints, they produce the monomial
  $\lcm(u_{\pi_3}, u_{\pi_4}) y_ix_\ell$ of the form (iii) that divides
  $wx_\ell y_j$. Then the $S$-polynomial reduces to zero.

  Suppose that $D \neq \emptyset$ and set $h_0 = \min D$. Then by
  \ref{lemmapaths} there exist two admissible paths $\tau_1$ and
  $\tau_2$ both with endpoints $i$ and $h_0$ or $j$ and $h_0$, such that
  $\tau_1$ is odd and $\tau_2$ is even or vice versa. In the first case,
  if $\tau_1$ and $\tau_2$ have endpoints $i$ and $h_0$, then the monomial
  $\lcm(u_{\tau_1}, u_{\tau_2}) y_ix_{h_0}$ divides $w$. If conversely
  $\tau_1$ and $\tau_2$ have endpoints $j$ and $h_0$, then the monomial
  $\lcm(u_{\tau_1}, u_{\tau_2}) y_jx_{h_0}$ divides $wx_\ell y_j$.
  In any case the $S$-polynomial reduces to zero.

  \noindent {\sf Case II:} $S$-polynomials of monomials from (iii) and (iv)
  and binomials from (i) and (ii)

  For the first case, it is enough to show that
  $S(u_{\pi_{ij}}b_{ij},\lcm(u_{\pi_{k\ell}},u_{\sigma_{k\ell}}) y_k x_\ell)$
  reduces to zero. The corresponding $S$-polynomial for an even path differs
  only by sign. Notice that
  \begin{equation} \label{Spair1}
   S(u_{\pi_{ij}}b_{ij},\lcm(u_{\pi_{k\ell}},u_{\sigma_{k\ell}})
    y_k x_\ell) = \frac{\lcm(u_{\pi_{ij}}x_iy_j,\lcm(u_{\pi_{k\ell}},
    u_{\sigma_{k\ell}}) y_k x_\ell)}{x_i y_j} y_i x_j.
  \end{equation}

  If both $x_i$ and $y_j$ do not divide the monomial $v=\lcm(u_{\pi_{k\ell}},
  u_{\sigma_{k\ell}}) y_k x_\ell$, then the $S$-polynomial reduce to $0$
  since it is a multiple of $v$. Therefore we may assume that $x_i$ or $y_j$
  divides $v$. Suppose that $x_i$ divides $v$. We claim that $y_j \nmid v$.
  Indeed, if $y_j | v$, then $j \leq k$. On the other hand, $i \geq \ell$
  because $x_i | v$. Therefore, $k < \ell \leq i < j$, a contradiction. Similarly one shows that if $y_j|v$, then $x_i \nmid v$.

  In the further discussion we assume that $x_i$ divides $v$. The case that
  $y_j | v$ can be treated similarly. Since $x_i | v$, it follows that
  $i \geq \ell$. Hence the $S$-polynomial \eqref{Spair1} can be rewritten as
  $\frac{\lcm(u_{\pi_{ij}},v)}{x_i} y_ix_j$. For simplicity we call this
  monomial $w$.

  First suppose that $i>\ell$, then
  $x_i | \lcm(u_{\pi_{k\ell}},u_{\sigma_{k\ell}})$ which implies that
  $i \in V(\pi_{k\ell}) \cup V(\sigma_{k\ell})$. Thus there is a path $\tau$
  from $\ell$ to $i$ which is part of $\pi_{k\ell} \cup \sigma_{k\ell}$.
  To simplify the notation we relabel its vertices as
  \[
    \tau: \ell=j_0,j_1,\dots,j_t=i.
  \]

  Set $j_{t(0)}=j_0$, $j_{t(1)}=\min\{j_h : j_h>\ell, h=1,\dots,t\}$ and
  $j_{t(2)}=\min\{j_h : j_h>\ell, h=t(1)+1,\dots,t\}$. Proceeding in this
  way we find a sequence of integers
  \[
    \ell = j_{t(0)} < j_{t(1)} < \cdots < j_{t(q)} = i
  \]
  and for each $1 \leq c \leq q$, the path
  $\tau_c: j_{t(c-1)},j_{t(c-1)+1},\dots,j_{t(c)}$ is admissible. Then
  \[
    S(u_{\pi_{ij}}b_{ij},\lcm(u_{\pi_{k\ell}},u_{\sigma_{k\ell}})
    y_k x_\ell) = (-1)^{t(q)-t(0)} \frac{w x_i y_\ell}{y_i x_\ell} +
    \sum_{c=1}^q (-1)^{t(q)-t(c)} v_{\tau_c} u_{\tau_c}
    h_{j_{t(c-1)}j_{t(c)}},
  \]
  where for every $c=1,\dots,q$,
  \[
    v_{\tau_c} = \frac{w x_i}{u_{\tau_c} y_i x_{j_{t(c)}} x_{j_{t(c-1)}}}.
  \]

  Notice that the first summand,
  $\displaystyle{\frac{w x_i y_\ell}{y_i x_\ell}}$, is a multiple of a
  monomial in $\mathcal G$ of the form (iii) by \ref{lemmapaths}. One
  checks that this is a standard expression for the given $S$-polynomial.

  Now suppose that $i=\ell$. If there is at least one vertex in
  $V(\pi_{k\ell}) \cup V(\sigma_{k\ell})$ that is bigger than $\ell$, then in
  light of \ref{lemmapaths} $w$ is a multiple of a monomial in
  $\mathcal G$ of the form (iii). If all the vertices in
  $V(\pi_{k\ell}) \cup V(\sigma_{k\ell})$ are less than or equal to $\ell$,
  then we label the vertices of $\pi_{ij}$ as $\pi_{ij}: i=i_0,i_1\dots,i_r=j$
  and choose $p = \min \{ h : i_h \geq j \}$. Thus
  $x_{i_p} | u_{\pi_{ij}}x_j$ and $w$ is a multiple of
  \[
    \lcm \left( \prod_{h \in W} y_h,
    \prod_{h=1}^{p-1} y_{i_h} \right) x_{i_p},
  \]
  that is a multiple of an element of $\mathcal G$ of the form (iv).

  \noindent {\sf Case III:} $S$-polynomials involving a monomial from (iv).

  It is enough to prove that $S(u_{\pi_{ij}}b_{ij},x_b \prod_{h \in W} y_h)$
  reduce to zero, where
  $W = V(\pi_{k\ell}) \cup V(\sigma_{k\ell}) \cup V(\tau_{ab}) \setminus \{b\}$
  as in the statement. Indeed, the corresponding $S$-polynomial in which we
  consider $g_{ij}$ instead of $b_{ij}$ differs only by sign. Moreover, the
  $S$-polynomial with a monomial of the form $y_b \prod_{h \in W} x_h$ can be
  treated similarly.

  As in the previous case, see \eqref{Spair1}, the only interesting cases
  are those in which $x_i$ or $y_j$ divides the monomial
  $u=x_b \prod_{h \in W} y_h$. Moreover, as before, one can show that only
  one of $x_i$ and $y_j$ can divide the monomial $u$. If $x_i$ divides $u$,
  then $i=b$. Hence the $S$-polynomial can be written as
  $\lcm(u_{\pi_{ij}},\prod_{h \in W} y_h) y_ix_j$. Suppose that
  $\pi_{ij}: i=i_0,i_1,\dots,i_r=j$ and define $p=\min \{ h : i_h \geq j\}$.
  It follows that $x_{i_p} | u_{\pi_{ij}}x_j$. Thus the $S$-polynomial
  is a multiple of the monomial
  \[
    \lcm \left( \prod_{h \in W} y_h, \prod_{h=0}^{p-1} y_{i_h} \right) x_{i_p},
  \]
  which is a multiple of an element of $\mathcal G$ of the form (iv). Hence
  the $S$-polynomial reduces to $0$.

  Suppose that $y_j$ divides $u$. Then the $S$-polynomial can be written as
  the monomial $w=\frac{\lcm(u_{\pi_{ij}},u)}{y_j} y_ix_j$. Let
  $m = \max \{h : h \in W \}$. First suppose that $j=m$. Set
  $H = \pi_{k\ell} \cup \sigma_{k\ell} \cup \tau_{ab}$. If $m$ belongs to
  some odd cycle $C$ of $H$, let $m' = \max \{h \in V(C) : h<m\}$ and call
  $\pi_{m'm}$ and $\sigma_{m'm}$ the two paths from $m'$ to $m$ that are
  contained in $C$. Notice that $\pi_{m'm}$ is odd and $\sigma_{m'm}$ is
  even or vice versa, and both are admissible. Hence they produce the monomial
  $\lcm(u_{\pi_{m'm}},u_{\sigma_{m'm}}) y_{m'} x_m$ of the form (iii) that
  divides $w$. On the other hand, if $m$ does not belong to any odd cycle of
  $H$, let $C$ be an odd cycle contained in $H$ such that
  \[
    d_H(m,C) = \min\{d_H(m,D) : D \text{ odd cycle of } H\},
  \]
  where $d_H(m,D)$ is the length of a shortest path connecting $m$ with a
  vertex of $D$. Call $\tau$ a shortest path in $H$ connecting $m$ to $C$.
  Let $p=\max \{h : h \in V(C)\}$, $p'=\max \{h \in V(C) : h < p\}$ and
  call $\pi_{p'p}$ and $\sigma_{p'p}$ the two paths from $p'$ to $p$
  contained in $C$. Notice that $\pi_{p'p}$ is odd and $\sigma_{p'p}$ is
  even or vice versa, and both are admissible. Hence $\pi_{p'p}$,
  $\sigma_{p'p}$ and $\tau$ produce the monomial
  $w' = x_m \prod_{h \in W'} y_h$, where
  $W' = V(\pi_{p'p}) \cup V(\sigma_{p'p}) \cup V(\tau) \setminus \{m\}$.
  The monomial $w'$ is of the form (iv) and divides $w$.

  Now suppose that $j<m$. We choose a path $\tau$ from $j$ to $m$ contained
  in $\pi_{k\ell} \cup \sigma_{k\ell} \cup \tau_{ab}$ and relabel its vertices
  as $\tau: j=j_0,j_1,\dots,j_t=m$. Set $j_{t(0)}=j_0$,
  $j_{t(1)}=\min \{j_h \in V(\tau): j_h>j, h=1,\dots,t \}$,
  $j_{t(2)}=\min \{ j_h \in V(\tau) : j_h>j, h=t(1)+1,\dots,t\}$. Proceeding
  in this way we define a sequence of integers
  $j=j_{t(0)} < j_{t(1)} < \cdots < j_{t(q)}=m$ and for each
  $1 \leq c \leq q$, the path
  $\tau_c: j_{t(c-1)},j_{t(c-1)+1},\dots,j_{t(c)}$ is admissible. Then
  \[
    S(u_{\pi_{ij}}b_{ij},x_b \prod_{h \in W} y_h) =
    \sum_{c=1}^q (-1)^{t(c-1)} v_{\tau_c} u_{\tau_c}
    h_{j_{t(c-1)}j_{t(c)}} + (-1)^{t(q)} \frac{w y_j x_m}{x_j y_m},
  \]
  where for every $c=1,\dots,q$, $v_{\tau_c} = \frac{w y_j}{u_{\tau_c} y_{j_{t(c-1)}} y_{j_{t(c)}} x_j}$.

  Notice, that the last summand, $\displaystyle{\frac{w y_j x_m}{x_j y_m} =
  \frac{\lcm (u_{\pi_{ij}},x_b \prod_{h \in W} y_h)}{y_m}} y_i x_m$. We
  call this monomial $v$. As we discussed before, in the case $j=m$, one can
  show that the monomial $v$ is divisible by a monomial of the form (iv). As
  before, one checks that this is a standard expression for the given
  $S$-polynomial.
\end{proof}

\begin{Corollary}
  \label{radical}
  Let  $G$ be a graph, and suppose $\chara(K)\neq 2$. Then $\Pi_{G}$
  is a radical ideal.
\end{Corollary}
\begin{proof}
  From \ref{Grobner} it follows that there is a monomial order
  $<$ such that $\ini_<(\Pi_G)$ is a squarefree
  monomial ideal. This together with \cite[Prop. 3.3.7]{HH} implies that $\Pi_G$ is a radical ideal.
\end{proof}

Now \ref{radical-L_G} follows.

\begin{proof}[Proof of \ref{radical-L_G}]
  Assume first that $\sqrt{-1}\in K$.
  Then $L_G$ and $\Pi_G$ arise from each other by a linear change of
  coordinates. Hence $L_G$ is a radical ideal if and only $\Pi_G$ is.

  Now suppose that $\sqrt{-1}\notin K$. We choose a field extension $L/K$
  with $\sqrt{-1}\in L$. Then $L_G\otimes_K L\subset
  L[x_1,\ldots,x_n,y_1,\ldots,y_n]$ is generated by the
  same binomials as $L_G$, and hence by the first part of the proof it follows
  that $L_G\otimes_K L$ is a radical ideal. Suppose $L_G$ is not radical.
  Then there exists $f\notin L_G$ such that $f^k\in L_G$ for  some $k$. It
  follows that $f^k\in L_G\otimes_K L$. It remains to show that
  $f\notin L_G\otimes_K L$. Suppose this is not the case. Let
  $\mu_f:T/L_G\rightarrow T/L_G$ be the $T/L_G$-module homomorphism induced
  by multiplication with $f$. Then $\Im(\mu_f\otimes_K L)=0$, because
  $f\in L_G\otimes_K L$. Since $L$ is a flat $K$-module, it follows that
  $\Im(\mu_f\otimes_K L)=(\Im \mu_f)\otimes_K L$, and hence
  $(\Im \mu_f)\otimes_K L=0$. Since $L$ is even faithfully flat over $K$,
  we conclude that $\Im \mu_f=0$. This implies that $f=0$, a contradiction.
\end{proof}

We note that in general for $L_G$ there is no term order such that $L_G$ has
a squarefree initial ideal with respect to this term order. For example, by using the routine \texttt{gfan} of \texttt{Macaulay2} \cite{Mac2}, it can be seen that $L_{C_3}$ has no squarefree Gr\"obner basis with respect to any term order.
\\[3mm] \indent Finally, we provide the proof of \ref{char2}.

\begin{proof}[Proof of \ref{char2}]
  If $G$ is bipartite then by $\sqrt{-1} \in K$ and \ref{binomial edge ideal} the
  radicality of $L_G$ is implied by the radicality of binomial
  edge ideals.

  It remains to show that if $G$ is not bipartite then $L_G$
  is not radical.
  It is enough to show that $L_G T_Y$ is not radical, since
 $Y$ contains only non-zero divisor modulo $L_G$ by \ref{nonzerodivisor}.
  Recall that via the change of variables
  $x_i \mapsto z_i = \frac{x_i}{y_i}$ for $i =1,\ldots, n$ we identify
  $T_Y$ with $K[z_1,\ldots, z_n,y_1^{\pm 1},\ldots, y_n^{\pm 1}]$.
  Thus ideal $L_GT_Y$ generated by the elements $z_iz_j+1$ for $\{i,j\}
  \in  E(G)$ in $T_Y$.
  We further transform $z_i \mapsto w_i:=1+z_i$ for $i=1,\ldots,n$.
  Then $L_GT_Y$ is generated by the elements $w_i+w_j+w_iw_j$ for $\{i,j\}\in  E(G)$ in
  $T_Y = K[w_1,\ldots,w_n,y_1^{\pm 1},\ldots,y_n^{\pm 1}]$.
  Since $G$ is non-bipartite, there exists a
  subgraph of $G$ which is an odd cycle, say $C_{m}$. We may assume that
  $V(C_m)=[m]$. Note that
  \begin{eqnarray}
    \label{cyclegenerators}
    \sum_{i=1}^{m-1}(w_i+w_{i+1}+w_iw_{i+1})+(w_1+w_m+w_1w_m)=\sum_{i=1}^{m-1}w_iw_{i+1}+w_1w_m,
  \end{eqnarray}
  since $\chara(K)=2$ and each $w_i$ appears twice in the sum on the left-hand side of the equation.
  It follows that $\sum_{i=1}^{m-1}w_iw_{i+1}+w_1w_m\in L_GT_Y$. We also
  have $w_{2i-1}w_{2i}+w_{2i}w_{2i+1}\in L_GT_Y$ for all
  $i=1,\ldots,(m-1)/2$, because
  \begin{eqnarray}
    \label{twosums}
    w_{2i-1}w_{2i}+w_{2i}w_{2i+1}&=&w_{2i+1}(w_{2i-1}+w_{2i}+w_{2i}w_{2i-1})\\
    \nonumber
    &+&w_{2i-1}(w_{2i}+w_{2i+1}+w_{2i}w_{2i+1}).
  \end{eqnarray}
  From \eqref{cyclegenerators} and \eqref{twosums} we deduce that
  $w_1w_m\in L_GT_Y$. By symmetry we also have
  $w_iw_{i+1}\in L_GT_Y$ for $i=1,\ldots,m-1$. This implies that
  $w_1+w_{m}\in L_GT_Y$ and $w_i+w_{i+1}\in L_GT_Y$ for $i=1,\ldots, m-1$.
  Hence $w_{i+1}^2\in L_GT_Y$ for $i=1,\ldots,m-1$, because
  $w_{i+1}^2=w_{i+1}(w_i+w_{i+1})+w_iw_{i+1}$. Similarly, $w_1^2\in L_GT_Y$.
  In order to conclude the proof of the theorem, we show that  $w_i\notin L_GT_Y$
  for all $i=1,\ldots, m$. Let $F$ be the quotient
  field of $K[y_1^{\pm 1},\ldots, y_n^{\pm 1}]$ and let $A=F[w_1,\ldots,w_n]/(w_{m+1},\ldots,w_n) \cong
  F[w_1,\ldots,w_m]$. It is enough
  to show that $w_i\notin L_GA$ for all $i=1,\ldots, m$. The above calculation have shown that $L_GA$
  is a graded ideal generated by the linear forms $w_1+w_{m}$ and $w_i+w_{i+1}$ for $i=1,\ldots,m-1$,
  and by the monomials $w_1w_m$ and  $w_iw_{i+1}$ for
  $i=1,\ldots, m-1$. Since $w_1+w_m=\sum_{i=1}^{m-1}(w_i+w_{i+1})$ we see that
  $\dim_F(L_GA)_1\leq m-1$, and hence not all
  $w_i$ belong to  $L_GA$. Say $w_1\not\in L_GA$.
  Since $w_i+w_{i+1}\in L_GA$ for $i=1,\ldots,m$ it then follows that
  $w_i\not\in L_GA$ for $i=1,\ldots,m$.
\end{proof}

\section[A primary decomposition of the ideal $L_G$]{A primary decomposition of the ideal $L_G$ for $\sqrt{-1}\notin K$}
\label{section2}

In this section, we give a primary decomposition of $L_G$ when
$\sqrt{-1}\notin K$. Let $H$ be  an arbitrary connected finite graph on
some vertex set. If $H$ is not bipartite, then we denote by
$\widetilde{H}$ the complete graph on the ground set of $H$. If $H$ is
bipartite, then the fact that $H$ is connected implies that the bipartition
of the ground set is uniquely defined. In this situation we denote by
$\widetilde{H}$ the complete graph on the ground set of $H$ with respect
to this bipartition.
Let $G$ be  a finite graph on the vertex set $[n]$. For
$S\subset [n]$ we set
\[
  Q_S(G)=(\{x_i,y_i\}_{i\in S}, I_{\widetilde{G}_1},\ldots, I_{\widetilde{G}_{c(S)}}),
\]
where $G_1,\ldots, G_{c(S)}$ are the connected components of
$G_{[n]\setminus S}$.
Notice that if $G$ is a connected bipartite graph,
then $Q_{\emptyset}(G)=I_{K_{m,n-m}}$ for some $m$,
and if $G$ is a connected non-bipartite graph, then $Q_{\emptyset}(G)=I_{K_n}$.
Moreover, note that if $G$ is bipartite, then by
\ref{binomial edge ideal}, it can be easily seen that the ideals
$Q_T(G)$ are isomorphic to the ideals $P_T(G)$ introduced in \cite{HHHKR}.
In particular, when $G$ is bipartite, $Q_{\emptyset}(G)$ is a determinantal
ideal. We denote the number of bipartite connected components of
$G_{[n]\setminus S}$ by $b(S)$. Here we consider $K_1$ as a bipartite graph.

\begin{Proposition}
  \label{height}
  Let $G$ be a graph on $[n]$ and let $S\subset [n]$. Then
  $\height Q_S(G)=|S|+n-b(S)$.
\end{Proposition}

\begin{proof}
  We may assume that $G_1,\ldots,G_{b(S)}$ are the bipartite connected
  components of $G$ and $G_{b(S)+1},\ldots,G_{c(S)}$ are the non-bipartite
  connected components of $G$. Let $n_j=|V(G_j)|$ for all $j=1,\ldots,c(S)$.
  Since the ideals $(x_i,y_i:i\in S)$ and $I_{\widetilde{G}_j}$ for
  $j=1,\ldots,c(S)$ are on pairwise disjoint sets of variables, it follows
  together with \ref{more} that
  \begin{eqnarray}
    \height Q_S(G) &=& \height (x_i,y_i:i\in S)+
      \sum_{j=1}^{b(S)}\height (I_{\widetilde{G}_j})+
      \sum_{j=b(S)+1}^{c(S)}\height (I_{\widetilde{G}_j})
      \nonumber \\
                &=& 2|S|+ \sum_{j=1}^{b(S)}(n_j-1)+
      \sum_{j=b(S)+1}^{c(S)} n_j
      \nonumber \\
                &=& |S|+(|S|+\sum_{j=1}^{b(S)}n_j+
      \sum_{j=b(S)+1}^{c(S)} n_j)-b(S)
      \nonumber \\
                &=& |S|+n-b(S),
      \nonumber
  \end{eqnarray}
  as desired.
\end{proof}

\begin{Proposition}
  \label{difficult}
  Let $\sqrt{-1}\notin K$, and let
  $I=\sum_{i}I_{K_{m_i,n_i-m_i}}+\sum_{j}I_{K_{t_j}}+U$,
  where $U$ is generated by variables. Suppose the ideals
  $I_{K_{m_i,n_i-m_i}}$, $I_{K_{t_j}}$ and $U$
  are defined on pairwise disjoint sets of variables.
  Then $I$ is a prime ideal.
  In particular, $Q_S(G)$ is a prime ideal for all $S\subset [n]$.
\end{Proposition}

\begin{proof}
  Let $I\subset R$,  where $R$ is the polynomial ring over $K$ in the
  variables which are needed to define $I$.
  Then $R'=R/U$ is a polynomial ring in the remaining variables of $R$
  which do not belong to $U$, and $I/U\subset R'$
  may be identified with the ideal
  $J=\sum_{i}I_{K_{m_i,n_i-m_i}}+\sum_{j}I_{K_{t_j}}\subset R'$.
  Since $R/I\iso R'/J$ it suffices to prove that $J$ is a prime ideal of
  $R'$. But this will follow from the
  following more general fact:

  \noindent {\sf Claim:} For $j=1,\ldots, m$ let $I_j$ be an ideal in the
  polynomial ring
    \[
      K[x_{11},\ldots,x_{1n_1},x_{21},\ldots,x_{2n_2},\ldots,x_{m1},\ldots,x_{mn_m} ]
    \]
    satisfying the following properties:
    \begin{enumerate}
      \item[(i)] the set of generators $\MG_j$ of  $I_j$ is a subset of $K[x_{j1},\ldots,x_{jn_j}]$;
      \item[(ii)] for all $j$ the coefficients of the elements of $\MG_j$ are $+1$ or $-1$;
      \item[(iii)] for any domain $B$ with $\sqrt{-1}\not\in B$ the ring $B[x_{j1},\ldots,x_{jn_j}]/({\MG}_j)B[x_{j1},\ldots,x_{jn_j}]$
	    is a domain and $\sqrt{-1}\not\in B[x_{j1},\ldots,x_{jn_j}]/(\MG_j)B[x_{j1},\ldots,x_{jn_j}]$.
    \end{enumerate}
    Then $I_1+\cdots +I_m$ is a prime ideal.
  \medskip

  Before proving the claim let us use this fact to show that $J$ is a prime ideal. In our particular case the ideals
  $I_j$ are the ideals $I_{K_{m_i,n_i-m_i}}$ and $I_{K_{t_j}}$. Let $\MH_i$ be the set of generators of
  $I_{K_{m_i,n_i-m_i}}$  as in \eqref{othergenerators} and \eqref{linearforms} and $\MG_j$ be the set of generators of
  $I_{K_{t_j}}$ as in \eqref{alsolinear} and \eqref{squaregenerators}. Clearly the conditions (i) and (ii)  are satisfied.
  Let $B$ be a domain with $\sqrt{-1}\not\in B$. We first show that $B[x_{j1},\ldots,x_{jn_j}]/(\MH_j)B[x_{j1},\ldots,x_{jn_j}]$
  and $B[x_{j1},\ldots,x_{jn_j}]/(\MG_j)B[x_{j1},\ldots,x_{jn_j}]$  are domains. As  in the proofs of \ref{primeness-K_{m,n-m}}
  and \ref{primeness-K_{n}} where it was shown that $I_{K_{m,n-m}}$ and $I_{K_n}$ are prime ideals, we need to show  that
  $z_1z_{m+1}+1$ generates a prime ideal in $B[z_1,z_{m+1}, y_1^{\pm 1}, \ldots, y_n^{\pm 1}]$, and that $z_1^2+1$ generates a
  prime ideal in  $B[z_1, y_1^{\pm 1}, \ldots, y_n^{\pm 1}]$. But this is obviously the case.
  Suppose $\sqrt{-1}\in B[x_{j1},\ldots,x_{jn_j}]/(\MH_j)B[x_{j1},\ldots,x_{jn_j}]$. Then there exists $f\in B[x_{j1},\ldots,x_{jn_j}]$
  such that $f^2+1\in J_j:=(\MH_j)B[x_{j1},\ldots,x_{jn_j}]$. Since $J_j$ is a graded ideal, all homogeneous components
  of $f^2+1$ belong to $J_j$. Therefore, if  $b$ is the constant term of $f$, then $b^2+1\in J_j$, which is only
  possible if $b^2+1=0$. However since $\sqrt{-1}\not\in B$, we obtain a contradiction.
  \medskip

  \noindent {\sf Proof of Claim:} We proceed by induction on $m$. The
     assertion is trivial for $m=1$.
     Let $B=K[x_{11},\ldots,x_{1n_1},x_{21},\ldots,x_{2n_2},\ldots,x_{(m-1)1},
     \ldots,x_{(m-1)n_{m-1}} ]/(\MG_1,\ldots, \MG_{m-1})$.
     Then by our induction $B$ is a domain and $\sqrt{-1}\not\in B$.
     Moreover, we have that
     $R/(I_1+\cdots+I_m)\iso B[x_{m1},\ldots,x_{jm_j}]/(\MG_m)$,
     and hence (iii) implies that $I_1+\cdots+I_m$ is a prime ideal.
\end{proof}

The main result of this section is the following theorem that provides a primary decomposition of $L_G$
when $\sqrt{-1}\notin K$ which however in most cases is redundant.

\begin{Theorem}
  \label{primarynew}
  Let $G$ be a graph on $[n]$. Suppose that $\sqrt{-1}\notin K$, and $P$ is a minimal prime ideal of $L_G$.
  Then $P=Q_S(G)$ for some $S\subset [n]$. In particular, $L_G=\Sect_{S\subset [n]}Q_S(G)$.
\end{Theorem}

To prove \ref{primarynew}, we need the following lemmata.

\begin{Lemma}
  \label{pair}
  Let $G$ be a connected graph on $[n]$, and let $P$ be a minimal prime ideal of $L_G$ containing
  a variable. Then there exists $k\in [n]$ such that $x_k,y_k\in P$.
\end{Lemma}

\begin{proof}
  If $G=K_2$, then $L_G$ is a prime ideal, and hence $P=L_G$. Since $L_G$ does not contain any variable,
  there is nothing to prove in this case. Now suppose that $G\neq  K_2$ and that $x_i\in P$.
  Let us first assume that $G$ is a bipartite graph on the vertex set $[n]$ with the bipartition $\{1,\ldots,m\}\cup \{m+1,\ldots,n\}$.
  Suppose on the contrary that there exists no $k\in [n]$ such that $x_k,y_k\in P$. We claim that $(x_1,\ldots,x_m,y_{m+1},\ldots,y_n)\subset P$.
  Given $j\in [m]$, there exists a path $i=i_0,i_1,\ldots,i_{2\ell}=j$. Here we used the fact that $G$ is
  connected. We show by induction on $\ell$ that $x_j\in P$.
  Suppose that $\ell=1$. Since $x_{i_0}x_{i_1}+y_{i_0}y_{i_1}\in P$ and $x_{i_0}\in P$ but $y_{i_0}\notin P$, it follows that $y_{i_1}\in P$. Similarly, since
  $x_{i_1}x_{i_2}+y_{i_1}y_{i_2}\in P$ and $y_{i_1}\in P$ but $x_{i_1}\notin P$, it follows that $x_{i_2}\in P$. Since $i_2,\ldots,i_{2\ell}=j$ is a path of length
  $2(\ell-1)$ and $x_{i_2}\in P$, by induction hypothesis it follows that $x_j\in P$. By a similar argument for any $j\in \{m+1,\ldots,n\}$, we  have $y_j\in P$.
  Hence we have
  \[
    L_G\subset I_{K_{m,n-m}}\subsetneq (x_1,\ldots,x_m,y_{m+1},\ldots,y_n)\subset P,
  \]
  which contradicts the assumption that $P$ is a minimal prime ideal of $L_G$ because $I_{K_{m,n-m}}$ is a prime ideal,
  see \ref{primeness-K_{m,n-m}}. Therefore, it follows that $x_k,y_k\in P$ for some $k$.
  Next assume that $G$ is a non-bipartite graph. Since $G$ is connected and non-bipartite, there exists $j\in [n]$ and an even
  path $i=i_0,i_1,\ldots,i_{2t}=j$, and an odd path $i=j_0,j_1,\ldots,j_{2s-1}=j$ in $G$ connecting $i$ and $j$.
  If there exists $\ell=i_0,\ldots,i_{2t}$ or $\ell=j_0,\ldots,j_{2s-1}$ with $x_\ell,y_\ell \in P$, then we are done.
  Otherwise, by an argument as in the bipartite case, we deduce from the generators $x_{i_r}x_{i_{r+1}}+y_{i_r}y_{i_{r+1}}$
  for all $r=0,\ldots,2t-1$, that $x_j\in P$. Similarly, we see that $y_j\in P$ by considering the generators attached to the odd path.
\end{proof}

\begin{Lemma}
  \label{novariable}
  Let $G$ be a connected graph on $[n]$, and let $P$ be a minimal prime ideal of $L_G$ which does not contain
  any variable. Suppose that $\sqrt{-1}\not\in K$. Then $P=I_{K_{n}}$ or $P=I_{K_{m,n-m}}$ for some $m$.
\end{Lemma}

\begin{proof}
  Suppose first that  $\widetilde{G}=K_{m,n-m}$. We have
  $L_G\subset I_{K_{m,n-m}}$ and claim that
  $L_GT_Y=I_{K_{m,n-m}}T_Y$. In the case that $G=K_{1,1}$ there is
  nothing to prove. Thus we may assume that $G$ has at least three vertices.
  It suffices to show that $I_{K_{m,n-m}}T_Y \subset L_GT_Y$. The ideal
  $L_GT_Y$ is generated by the elements $z_iz_j+1$, where $\{i,j\}\in E(G)$.
  We will show that
  $z_i-z_j\in L_GT_Y$ for all $1\leq i<j\leq m$ and for all
  $m+1\leq i<j\leq n$. This
  together with \eqref{othergenerators} and \eqref{linearforms} then will
  imply that indeed
  $I_{K_{m,n-m}}\subset L_GT_Y$. Let $1\leq i<j\leq m$
  (the case $m+1\leq i<j\leq n$ can be treated similarly). Since $G$ is
  connected, there exists a path $i=i_0,i_1,\ldots,i_{2s}=j$ in $G$. We
  have $z_i-z_j=\sum_{t=1}^{s}(z_{i_{2t-2}}-z_{i_{2t}})$.
  So it suffices to prove each of the summands
  $z_{i_{2t-2}}-z_{i_{2t}}\in L_GT_Y$. Thus we may as well assume that $s=1$.
  We have
  $z_i-z_j=z_{i_0}-z_{i_2}=z_{i_0}(z_{i_1}z_{i_2}+1)-z_{i_2}(z_{i_0}z_{i_1}+1)$
  which is an element of $L_GT_Y$. It proves the claim that
  $L_GT_Y=I_{K_{m,n-m}}T_Y$. It follows that $I_{K_{m,n-m}}T_Y \subset PT_Y$,
  and hence $I_{K_{m,n-m}}\subset P$ since $P$ is a prime ideal and $y\notin P$.
  Finally, since $P$ is a minimal prime ideal of $L_G$, and $I_{K_{m,n-m}}$
  is a prime ideal by \ref{primeness-K_{m,n-m}}, we have $P=I_{K_{m,n-m}}$.
  Next we consider the case that $\widetilde{G}=K_{n}$. Similarly as in
  bipartite case we have $L_G\subset I_{K_{n}}$ and claim that
  $L_GT_Y=I_{K_n}T_Y$. It suffices to show that $I_{K_n}T_Y\subset L_GT_Y$.
  The ideal $L_GT_Y$ is generated by the elements $z_iz_j+1$,
  where $\{i,j\}\in E(G)$. We will show that $z_i-z_j\in L_GT_Y$ for all
  $1\leq i<j\leq n$. This together with \eqref{alsolinear} and
  \eqref{squaregenerators}
  then will imply that indeed $I_{K_{n}}T_Y\subset L_GT_Y$.
  Let $1\leq i<j\leq n$. Since $G$ is non-bipartite, there exists an even
  walk (not necessarily a path)
  in $G$ connecting $i$ and $j$. As in the bipartite case, we deduce from this
  fact that $z_i-z_j\in L_GT_Y$.
  As in the previous case it follows that $I_{K_n}\subset P$, and hence $P=I_{K_n}$
  since by \ref{primeness-K_{n}} $I_{K_n}$ is  a prime ideal.
\end{proof}

\begin{proof}[Proof of \ref{primarynew}]
  We prove the theorem by induction on the number of vertices of  $G$. If $|V(G)|=2$, then $G=K_{1,1}$ and $L_G=Q_{\emptyset}(G)$.
  Now suppose that $|V(G)|>2$, and let $G_1,\ldots,G_t$ be the connected components of $G$.
  Suppose first that  $t>1$. For $i=1,\ldots,t$ let $P_i$ be a minimal prime ideal of $L_{G_i}$ which is contained in $P$.
  Then $\sum_{i=1}^tP_i\subset P$. Since $|V(G_i)|<|V(G)|$ for all $i$, our induction hypothesis implies that there exists subset
  $S_i$ such that $P_i=Q_{S_i}(G_i)$. Therefore, $\sum_{i=1}^tQ_{S_i}(G_i)\subset P$. Since $\sum_{i=1}^tQ_{S_i}(G_i)=Q_S(G)$
  where $S=\bigcup_{i=1}^t S_i$, it follows that $Q_S(G)\subset P$. By \ref{difficult}, $Q_S(G)$ is a prime ideal, and hence
  $P=Q_S(G)$ because $P$ is a minimal prime ideal of $L_G$.
  Next suppose that $t=1$. If $P$ does not contain any variable, then by \ref{novariable} either $P=I_{K_n}$ or $P=I_{K_{m,n-m}}$
  for suitable $m$. In either case, $P=Q_{\emptyset}(G)$. If $P$ contains a variable, then by \ref{pair}, there exists $k$ such that
  $x_k,y_k\in P$. Let $\overline{P}=P/(x_k,y_k)$. Then $\overline{P}$ is a minimal prime ideal of $L_{G_{[n]\setminus \{k\}}}$. By induction
  hypothesis, there exists $\overline{S}\subset [n]\setminus \{k\}$ such that $\overline{P}=Q_{\overline{S}}(G_{[n]\setminus \{k\}})$.
  It follows that $P=Q_{S}(G)$ where $S=\overline{S}\cup \{k\}$.
  Since by \ref{difficult} all $Q_S(G)$ are prime ideals and each $Q_S(G)$ contains $L_G$, the identity $L_G=\Sect_{S\subset [n]}Q_S(G)$
  follows from the first part of the theorem and the fact that $L_G=\sqrt{L_G}$, see \ref{radical}.
\end{proof}

Now by \ref{height} and \ref{primarynew}, we get the following corollary about the Krull dimension of $T/L_G$.

\begin{Corollary}
  \label{dim}
  Let $G$ be a graph on $[n]$, and let $\sqrt{-1}\notin K$. Then
  \[
    \dim (T/L_G)=\max \{n-|S|+b(S):S \subset [n]\}.
  \]
  In particular, $\dim (T/L_G)\geq n+b$ where $b$ is the number of
  bipartite connected components of $G$.
  Moreover, if $L_G$ is unmixed, then $\dim (T/L_G)=n+b$.
\end{Corollary}

\begin{proof}
  The first equality follows easily from \ref{height} and \ref{primarynew}. Moreover, it follows from \ref{primarynew} that
  $\dim (T/L_G) \geq \dim (T/Q_{\emptyset}(G))$.
\end{proof}

Note that the lower bound given in \ref{dim} might be strict. For example let $G$ be the graph which is shown in Figure~\ref{butterfly}. Then
$\dim (T/L_G)=6$, while $n=5$ and $b=0$. On the other hand, $\dim (T/L_G)=n+b$ does not in general imply that $L_G$ is unmixed. For instance,
$\dim (T/L_{K_{2,2}})=5$ and in this case we have $n=4$ and $b=1$,
but $L_{K_{2,2}}$ is not unmixed (see also \ref{cycle}).

\begin{figure}[hbt]
\begin{center}
\psset{unit=0.6cm}
\begin{pspicture}(-1,-2)(4,2)
\rput(-0.1,-1){
\rput(0,0){$\bullet$}
\rput(1.5,1){$\bullet$}
\rput(0,2){$\bullet$}

\rput(3,0){$\bullet$}
\rput(3,2){$\bullet$}
\psline(0,0)(0,2)
\psline(1.5,1)(0,2)
\psline(1.5,1)(0,0)
\psline(1.5,1)(3,0)
\psline(3,2)(1.5,1)
\psline(3,0)(3,2)
}
\end{pspicture}
\end{center}
\caption{}\label{butterfly}
\end{figure}

Finally, we give a proof of the geometric consequences of \ref{primarynew}.

\begin{proof}[Proof of \ref{geometriclovasz}]
  It is a direct consequence of \ref{primarynew} that $\OR_2^K(G)$ is the
  union of the vanishing sets of the $Q_S(\overline{G})$ for $S \subseteq [n]$.
  Assume that the deletion of $S$ from $\overline{G}$ has a non-singleton non-bipartite
  component $H$. Then for all vertices $i$ from $H$ the element $x_i^2+y_i^2$ is
  in $Q_S(\overline{G})$. By $\sqrt{-1} \not\in K$ this implies that $x_i = y_i =0$
  in the vanishing set of $Q_S(\overline{G})$. But this implies that the vanishing set
  of $Q_S(\overline{G})$ is contained in the vanishing set of $Q_{S \cup V(H)}(\overline{G})$.
  Thus $\OR_2^\RR(G)$ is the union of the vanishing sets of $Q_{S}(\overline{G})$ where
  $S \subseteq [n]$ such that the deletion of $S$ from $\overline{G}$ creates only
  singleton or bipartite connected components.
  Now a direct translation of the equations expressing $Q_S(\overline{G})$ yields
  the formulation of the vanishing set given in \ref{geometriclovasz}.
  Finally, by \ref{difficult} all ideals $Q_S(\overline{G})$ are prime.
\end{proof}

\section{The minimal prime ideals of $L_G$ when $\sqrt{-1}\notin K$}
\label{section3}

Now we want to find all the minimal prime ideals of the ideal $L_G$ when $\sqrt{-1}\notin K$.
For this purpose, we need the following proposition and some other definitions.

\begin{Proposition}
  \label{minimalprimes1}
  Let $G$ be a graph on $[n]$, and let $S,W\subset [n]$. Let $H_1,\ldots,H_t$ and $G_1,\ldots,G_s$
  be the connected components of $G_{[n]\setminus S}$ and
  $G_{[n]\setminus W}$, respectively. Then $Q_S(G)\subset Q_W(G)$ if and only if $S\subset W$ and for all $i\in [t]$ with $|V(H_i)|>1$, there exists $j\in [s]$ such that $V(H_i)\setminus W\subset V(G_j)$, and if $H_i$ is bipartite
  (resp. non-bipartite), then $G_j$ is also bipartite (resp. non-bipartite).
\end{Proposition}
\begin{proof}
  For every $A\subset [n]$, let $U_A=(x_i,y_i:i\in A)$. Then $Q_S(G)=(U_S,I_{\widetilde{H}_1},\ldots,I_{\widetilde{H}_t})$ and
  $Q_W(G)=(U_W,I_{\widetilde{G}_1},\ldots,I_{\widetilde{G}_s})$. One has $Q_S(G)\subset Q_W(G)$ if and only if $S\subset W$ and
  $(U_W,I_{\widetilde{H}_1},\ldots,I_{\widetilde{H}_t})\subset (U_W,I_{\widetilde{G}_1},\ldots,I_{\widetilde{G}_s})$. For all $i=1,\ldots,t$ let $I'_{\widetilde{H}_i}$ be the ideal generated by those generators of $I_{\widetilde{H}_i}$ which belong to $R=K[x_i,y_i:i\in [n]\setminus W]$. Then $(U_W,I_{\widetilde{H}_1},\ldots,I_{\widetilde{H}_t})=(U_W,I'_{\widetilde{H}_1},\ldots,I'_{\widetilde{H}_t})$. It follows that $Q_S(G)\subset Q_W(G)$ if and only if $S\subset W$ and $(U_W,I'_{\widetilde{H}_1},\ldots,I'_{\widetilde{H}_t})\subset (U_W,I_{\widetilde{G}_1},\ldots,I_{\widetilde{G}_s})$. The latter inclusion
  holds if and only if $(I'_{\widetilde{H}_1},\ldots,I'_{\widetilde{H}_t})\subset (I_{\widetilde{G}_1},\ldots,I_{\widetilde{G}_s})$, since the generators of the
  ideals $(I'_{\widetilde{H}_1},\ldots,I'_{\widetilde{H}_t})$ and $(I_{\widetilde{G}_1},\ldots,I_{\widetilde{G}_s})$ belong to $R$.
  Now suppose $S\subset W$. It is
  enough to show that the following conditions are equivalent:
  \begin{enumerate}
    \item[(a)] For all $i\in [t]$ with $|V(H_i)|>1$, there exists $j\in [s]$ such that $V(H_i)\setminus W\subset V(G_j)$, and if $H_i$ is bipartite (resp.
      non-bipartite), then $G_j$ is also bipartite (resp. non-bipartite).
    \item[(b)] $(I'_{\widetilde{H}_1},\ldots,I'_{\widetilde{H}_t})\subset (I_{\widetilde{G}_1},\ldots,I_{\widetilde{G}_s})$.
  \end{enumerate}
  By definitions and notation, (a)\implies (b) is clear.
  For the converse, let $i\in [t]$ with $|V(H_i)|>1$, and let $k\in V(H_i)\setminus W$. Then $k\in V(G_j)$ for some $j\in [s]$. We claim that $V(H_i)\setminus W\subset V(G_j)$.
  If $V(H_i)\setminus W=\{k\}$, there is nothing to prove. So we may assume that $|V(H_i)\setminus W|\geq 2$. Suppose that there is an element $l\in
  V(H_i)\setminus W$ such that $l\neq k$ and $l\notin V(G_j)$. Then there exists $r\in [s]$ with $r\neq j$ such that $l\in V(G_r)$. We may assume that $k<l$.
  First suppose that $H_i$ is a bipartite graph on $A_1\cup A_2$. Since $V(H_i)\setminus W$ is non-empty, it follows that each connected component of
  ${(H_i)}_{[n]\setminus W}$ is a connected component of $G_{[n]\setminus W}$, and since $H_i$ is bipartite, each of its components is bipartite as well. Hence
  \begin{eqnarray}
    \label{bipartitecomponent}
    V(H_i)\setminus W\subset \bigcup_{d=1\atop G_d \text{ bipartite}}^s V(G_d).
  \end{eqnarray}
  In particular, $G_j$ and $G_r$ are bipartite.
  If $k,l\in A_1$ or $k,l\in A_2$, then
  $g_{kl}=x_ky_l-x_ly_k\in I'_{\widetilde{H}_i}$. Hence by the assumption~(b),
  $g_{kl}\in (I_{\widetilde{G}_1},\ldots,I_{\widetilde{G}_s})$. Thus
  \begin{eqnarray}
    \label{sum1}
    g_{kl}=\sum_{t=1}^pr_tq_t,
  \end{eqnarray}
  where $r_t\in T$ and each $q_t$ is a generator of
  $(I_{\widetilde{G}_1},\ldots,I_{\widetilde{G}_s})$.
  Now, for $m\neq k,l$, we put all variables $x_m$ and $y_m$
  equal to zero in the equality \eqref{sum1} and denote by
  $\overline{q}_t$ the image of $q_t$ under this reduction.
  Then all $\overline{q}_t$ which are different
  from the binomials $f_{kl}$, $g_{kl}$, $h_k$ and $h_l$, are zero.
  Since $k$ and $l$ are contained in the different components
  $G_j$ and $G_r$, respectively, it follows that
  $\overline{q}_t\neq f_{kl},g_{kl}$. Also, since $k$ and $l$ belong
  to the bipartite components $G_j$ and $G_r$, respectively, it follows that
  $\overline{q}_t\neq h_{k},h_{l}$. Thus we see that after this reduction the
  right-hand side of \eqref{sum1} is zero while the left-hand side is non-zero,
  a contradiction. If $k\in A_1$ and $l\in A_2$, then
  $f_{kl}=x_kx_l+y_ky_l\in I'_{\widetilde{H}_i}$. Hence by the assumption
  (b), $f_{kl}\in (I_{\widetilde{G}_1},\ldots,I_{\widetilde{G}_s})$. Then,
  similar to the previous case, we get a contradiction.
  Next, suppose that $H_i$ is non-bipartite.
  So $g_{kl}=x_ky_l-x_ly_k\in I'_{\widetilde{H}_i}$, and hence by
  assumption (b),
  $g_{kl}\in (I_{\widetilde{G}_1},\ldots,I_{\widetilde{G}_s})$. Thus
  $g_{kl}=\sum_{t=1}^pr_tq_t$, where $r_t\in T$ and each $q_t$ is a
  generator of $(I_{\widetilde{G}_1},\ldots,I_{\widetilde{G}_s})$. Now,
  as before, for $m\neq k,l$, we put all variables $x_m$ and $y_m$ equal
  to zero in this equality.
  After reduction it follows that $g_{kl}$ can be written as
  $g_{kl}=r(x_k^2+y_k^2)+r'(x_l^2+y_l^2)$ for some polynomials
  $r,r'\in T$, which is a contradiction.
  Thus we see that indeed $V(H_i)\setminus W\subset V(G_j)$ which proves
  the claim.
  \medskip
  By \eqref{bipartitecomponent}, it also follows that if $H_i$ is bipartite,
  then $G_j$ is bipartite. Now, we show that if $H_i$ is non-bipartite,
  then $G_j$ is also non-bipartite. Indeed, if $H_i$ is non-bipartite, then
  $h_k=x_k^2+y_k^2\in I'_{\widetilde{H}_i}$, and hence  by the assumption (b),
  $h_{k}\in (I_{\widetilde{G}_1},\ldots,I_{\widetilde{G}_s})$. Thus
  \begin{eqnarray}
    \label{sum2}
    h_k=\sum_{t=1}^pr_tq_t,
  \end{eqnarray}
  where $r_t\in T$ and each $q_t$ is a generator of
  $(I_{\widetilde{G}_1},\ldots,I_{\widetilde{G}_s})$. If $h_k\neq q_t$ for
  all $t=1,\ldots,p$, then by substituting  all variables $x_m$ and $y_m$
  equal to zero for $m\neq k$, as before, we get $h_k=0$, which is a
  contradiction. It follows
  that $h_k=q_t$ for some $t=1,\ldots,p$. Therefore, $h_k$ is a generator
  of $I_{\widetilde{G}_j}$, and hence $G_j$ is a non-bipartite graph, too.
\end{proof}

Let $G$ be a graph on $[n]$. Then the vertex $i\in [n]$ is said to be a
\textit{cut point} of $G$ if $G_{[n]\setminus \{i\}}$ has more connected
components than $G$. Moreover, we call a vertex $i\in [n]$ a
\textit{bipartition point} of $G$ if $G_{[n]\setminus \{i\}}$ has more
bipartite connected components than $G$.
Let $\MM(G)$ be the set of all sets $S\subset [n]$ such that each $i\in S$ is
either a cut point or a bipartition point of the graph $G_{([n]\setminus
S)\cup \{i\}}$. In particular, $\emptyset\in \MM(G)$.
Now we are ready to determine all minimal prime ideals of $L_G$ in the case
$\sqrt {-1}\notin K$.

\begin{Theorem}
  \label{minimalprimes2}
  Let $G$ be a graph on $[n]$, and $S\subset [n]$. Suppose $\sqrt {-1}\notin K$. Then $Q_S(G)$
  is a  minimal prime ideal of $L_G$ if and only if $S\in \MM(G)$.
\end{Theorem}

\begin{proof}
  To prove the theorem we apply a similar strategy as used in the proof of \cite[Cor.~3.9]{HHHKR}.
  Assume first that $Q_S(G)$ is a minimal prime ideal of $L_G$. Suppose that $S\neq \emptyset$ and let
  $G_1,\ldots,G_r$ be the connected components of $G_{[n]\setminus S}$. We show that $S\in \MM(G)$. Let
  $i\in S$ and $T=S\setminus \{i\}$. Now we show that $i$ is either a cut point or a bipartition point of
  the graph $G_{[n]\setminus T}$.
  If $i$ is not adjacent to any vertex of $G_1,\ldots,G_r$, then the connected components of $G_{[n]\setminus T}$
  are $G_1,\ldots,G_r$ together with the isolated vertex $i$. So by \ref{minimalprimes1},
  $Q_T(G)\subsetneq Q_S(G)$. This is impossible, since by the assumption $Q_S(G)$ is a minimal prime ideal
  of $L_G$, and by \ref{difficult}, $Q_T(G)$ is a prime ideal containing $L_G$. So there
  exist some connected components of $G_{[n]\setminus S}$, say $G_1,\ldots,G_k$, which have at least one
  vertex adjacent to $i$. Then $G'_1,G_{k+1},\ldots,G_r$ are the connected components of $G_{[n]\setminus T}$,
  where $G'_1$ is the induced subgraph of $G_{[n]\setminus T}$ on $(\bigcup_{j=1}^kV(G_j))\cup \{i\}$.
  First suppose that $k=1$. Then $i$ is not a cut point of $G_{[n]\setminus T}$. If $G'_1$ is bipartite, then $G$ is also bipartite.
  Therefore, by \ref{minimalprimes1}, we have $Q_T(G)\subsetneq Q_S(G)$, contradicting the fact
  that $Q_S(G)$ is a minimal prime ideal of $L_G$.
  Similarly, if $G'_1$ and $G_1$ are both  non-bipartite, we get a contradiction. If $G'_1$ is non-bipartite
  and $G_1$ is bipartite, then $i$ is a bipartition point of $G_{[n]\setminus T}$.
  Next suppose that $k\geq 2$. Then clearly $i$  is a cut point of $G_{[n]\setminus T}$.
  Thus we have that indeed $S\in \MM(G)$.
  Conversely, since $Q_{\emptyset}(G)$ does not contain any variable, it is not contained in any other $Q_T(G)$.
  So $Q_{\emptyset}(G)$ is a minimal prime ideal of
  $L_G$. Now let $\emptyset\neq S\in \MM(G)$ and let $G_1,\ldots,G_r$ be the connected components of
  $G_{[n]\setminus S}$. Suppose that $Q_S(G)$ is not a minimal prime ideal of $L_G$. Then by \ref{difficult}
  and \ref{minimalprimes1}, there exists some $T\subsetneq S$ such that $Q_T(G)\subsetneq Q_S(G)$.
  Let $i\in S\setminus T$. Then $i$ is either a cut point or a bipartition point
  of $G_{([n]\setminus S)\cup \{i\}}$, since $S\in \MM(G)$.
  If $i$ is a cut point of $G_{([n]\setminus S)\cup \{i\}}$, then by a similar argument as in the first part of
  the proof, $G'_1,G_{k+1},\ldots,G_r$ are the connected components of $G_{([n]\setminus S)\cup \{i\}}$,
  where $k\geq 2$ and $G'_1$ is the induced subgraph of $G_{([n]\setminus S)\cup \{i\}}$ on
  $(\bigcup_{j=1}^kV(G_j))\cup \{i\}$. Thus $G_{[n]\setminus T}$ has a connected component $H$ which contains
  $G'_1$ as an induced subgraph. So $\bigcup_{j=1}^kV(G_j)\subset V(H)\setminus S$, which contradicts the fact
  that $Q_T(G)\subset Q_S(G)$, by \ref{minimalprimes1}.
  If $i$ is a bipartition point but not a cut point of $G_{([n]\setminus S)\cup \{i\}}$, then by a similar
  argument as in the first part of the proof, $G'_1,G_2,\ldots,G_r$ are the connected components of
  $G_{([n]\setminus S)\cup \{i\}}$, where $G'_1$ is the induced subgraph of $G_{([n]\setminus S)\cup \{i\}}$
  on $V(G_1)\cup \{i\}$ such that $G'_1$ is non-bipartite and $G_1$ is bipartite. Thus $G_{[n]\setminus T}$ has
  a connected component $H$ which contains $G'_1$ as an induced subgraph, and hence $H$ is also a non-bipartite graph.
  Moreover, $V(G_1)\subset V(H)\setminus S$. Hence by \ref{minimalprimes1},
  $V(G_1)=V(H)\setminus S$, since $Q_T(G)\subset Q_S(G)$. Again by applying \ref{minimalprimes1} we obtain
  a contradiction, since $H$ is non-bipartite but $G_1$ is bipartite.
\end{proof}

Let $G$ be the graph which is shown in Figure~\ref{cutpoints}, and let $T_1=\{4\}$, $T_2=\{4,5\}$ and $T_3=\{2,6\}$.
Note that the vertex $4$ is a cut point but not a bipartition point of the graph $G_{([7]\setminus T_2)\cup \{4\}}$,
while $5$ is a bipartition point but not a cut point of the graph $G_{([7]\setminus T_2)\cup \{5\}}$. By applying
\ref{minimalprimes2}, we have that the ideals $Q_{T_1}(G)$, $Q_{T_2}(G)$ and $Q_{T_3}(G)$ are minimal prime
ideals of $L_G$ when $\sqrt {-1}\notin K$. But $Q_{T_4}(G)$, where $T_4=\{3,7\}$, is not a minimal prime ideal of $L_G$,
because the vertex $7$ is neither a cut point nor a bipartition point of the graph $G_{([7]\setminus T_4)\cup \{7\}}$.

\begin{figure}[hbt]
\begin{center}
\psset{unit=0.6cm}
\begin{pspicture}(-0.5,-0.5)(7,2.5)
\psdots(0,0)(1.5,1)(0,2)(3.5,1)(5,0)(6.5,1)(5,2)
\psline(0,0)(0,2)
\psline(1.5,1)(0,2)
\psline(1.5,1)(0,0)
\psline(1.5,1)(3.5,1)
\psline(3.5,1)(5,0)
\psline(5,2)(3.5,1)
\psline(5,0)(5,2)
\psline(5,0)(6.5,1)
\psline(6.5,1)(5,2)
\uput[120](3.5,1){$4$}
\uput[90](5,2){$5$}
\uput[140](0,-0.45){$2$}
\uput[120](1.85,1){$3$}
\uput[120](0,2){$1$}
\uput[140](5.45,-0.85){$6$}
\uput[90](6.5,1){$7$}
\end{pspicture}
\end{center}
\caption{} \label{cutpoints}
\end{figure}

In the next corollary we determine when $L_G$ is a prime ideal for $\sqrt{-1}\notin K$.

\begin{Corollary}
  \label{prime}
  Let $K$ be a field such that $\chara(K) = 0$ or $\chara(K) \not\equiv 1,2 ~\mod 4$.
  Then the ideal $L_G$ is prime if and only if $G$ is a disjoint union of edges and isolated vertices.
\end{Corollary}

\begin{proof}
  Let $G$ be a disjoint union of edges and isolated vertices.
  It suffices to prove that $L_G$ is prime in the case of algebraically closed fields $K$.
  The ideal $L_G$ is the sum of ideals of the
  form $(x_ix_j+y_iy_j)$ for $i \neq j$ which are defined on
  pairwise disjoint sets of variables, and hence $T/L_G$ is a tensor product of copies of
  $K[x_i,x_j,y_i,y_j]/(x_ix_j+y_iy_j)$ for $i \neq j$ and a polynomial ring. Since $x_ix_j+y_iy_j$
  for $i \neq j$ is irreducible
  over any field, each factor is a domain. Now it follows from \cite[Prop. 5.17]{Mil} that $T/L_G$ is a domain and
  hence that $L_G$ is prime.

  To show that if $L_G$ is prime, then $G$ is a disjoint union of edges and isolated vertices, we may assume
  that $K$ is a prime field such that $\chara(K) = 0$ or
  $\chara(K) \not\equiv 1,2~\mod~4$, equivalently $K$ is a prime field such that
  $\sqrt{-1} \not\in K$.
  If $L_G$ is a prime ideal, then $L_G=Q_{\emptyset}(G)$, since
  by \ref{minimalprimes2}, $Q_{\emptyset}(G)$
  is a minimal prime ideal of $L_G$. This implies that $\MM(G)=\{\emptyset\}$,
  by \ref{minimalprimes2}. Let $H$ be a
  connected component of $G$. We show that $H=K_2$ or a single vertex. First suppose that $H$ is
  not a complete graph. Then there exists a
  minimal non-empty subset $S$ of $V(H)$ with the property that
  $H_{V(H)\setminus S}$ is a disconnected graph.
  It follows that each element $i$ of $S$ is a cut point of
  $H_{([n]\setminus S)\cup \{i\}}$, and hence a cut point of
  $G_{([n]\setminus S)\cup \{i\}}$. Therefore, by \ref{minimalprimes2},
  $S\in \MM(G)$, which contradicts the fact that
  $\MM(G)=\{\emptyset\}$, and hence $H$ is complete. Let $H=K_m$ where
  $V(H)=[m]$ and $m\geq 3$.
  Then $S'=[m]\setminus \{1,2\}\in \MM(G)$, since each element $i$ of $S'$ is
  a bipartition point of the graph
  $G_{([n]\setminus S')\cup \{i\}}=K_3$. Therefore, we get a contradiction,
  and hence $H=K_2$ or a single vertex, as desired.
\end{proof}

Now it is a simple exercise in graph theory to provide the proof of
\ref{lovaszcor}.

\begin{proof}[Proof of \ref{lovaszcor}]
  If $G$ is $(n-2)$-connected if and only if for each vertex $i$ there is at most
  one vertex $j$ such that $\{i,j\} \not\in E(G)$. Thus $G$ is $(n-2)$-connected if
  and only if $\overline{G}$ consists of a set of disjoint edges and isolated vertices.
  Now the assertion follows from \ref{prime}.
\end{proof}

Using the correspondence between the set of minimal prime ideals of $L_G$ and
the set $\MM(G)$ given in \ref{minimalprimes2}, we get immediately the
following criterion for unmixedness of the ideal $L_G$ when
$\sqrt{-1}\notin K$.

\begin{Corollary}
  \label{unmixed}
  Let $G$ be a graph with $b$ bipartite connected components, and let
  $\sqrt{-1}\notin K$. Then $L_G$ is unmixed if and only if $b(S)=|S|+b$
  for every $\emptyset\neq S\in \MM(G)$.
\end{Corollary}

\begin{proof}
  The ideal $L_G$ is unmixed if and only if all the minimal prime ideals of
  $L_G$ have the same height.
  By \ref{minimalprimes2}, this is equivalent to say that for all
  $\emptyset\neq S\in \MM(G)$,
  \begin{eqnarray}
    \label{sameheight}
    \height (Q_S(G))=\height (Q_{\emptyset}(G)).
  \end{eqnarray}
  By \ref{height}, this is the case if  and only if for every
  $\emptyset\neq S\in \MM(G)$, we have $b(S)=|S|+b$.
\end{proof}

As an application of the above criterion, we determine when $L_G$ is unmixed
for some special classes of graphs. We denoted by $C_n$ the $n$-cycle.

\begin{Corollary}
  \label{cycle}
  Let $n\geq 3$ be an integer, and $\sqrt{-1}\notin K$. Then $L_{C_n}$ is
  unmixed if and only if $n$ is odd.
\end{Corollary}

\begin{proof}
  First suppose that $n=2k$ for some $k\geq 2$, so that in this case $C_n$
  is a bipartite graph on $[n]$.
  Clearly $S=\{1,3\}\in \MM(C_{n})$. So $b(S)=2$, but $|S|+1=3$. Therefore,
  by \ref{unmixed}, $L_{C_{n}}$ is not  unmixed.
  Now suppose that $n=2k+1$ for some $k\geq 1$. Let
  $\emptyset \neq S\subset [n]$. Then $S\in \MM(C_n)$
  if and only if $S$ does not contain any two adjacent vertices of $C_n$.
  So, removing each of the elements of $S$ increases the number of bipartite
  connected components of the
  corresponding graph by one. Therefore, $b(S)=c(S)=|S|$, and hence $L_{C_n}$
  is unmixed by \ref{unmixed},
  since $C_n$ is non-bipartite in this case which means $b=0$.
\end{proof}

\begin{Corollary}
  \label{complete}
  Let $n\geq 2$ be an integer, and $\sqrt{-1}\notin K$. Then $L_{K_n}$ is unmixed if and only if $n=2$ or $3$.
\end{Corollary}

\begin{proof}
  If $n=2$ or $3$, then by \ref{prime} and \ref{cycle}, $L_{K_n}$ is unmixed.
  If $n\geq 4$, then $S=\{1,\ldots,n-2\}\in \MM(K_n)$, since for all
  $i=1,\ldots,n-2$, the vertex $i$ is a bipartition point
  of the graph ${(K_n)}_{{([n]\setminus S)\cup \{i\}}}$, which is a
  $3$-cycle. Thus $|S|=n-2$ and $b(S)=1$,
  and hence $b(S)\neq |S|$, since $n\geq 4$. Therefore, by \ref{unmixed},
  $L_{K_n}$ is not unmixed, since $K_n$ is non-bipartite.
\end{proof}

\section*{Acknowledgement}
We thank Bernd Sturmfels for interesting discussions and hints.

\end{document}